\documentclass[11pt,reqno]{amsart} 
\usepackage[T1]{fontenc}
\usepackage[utf8]{inputenc}
\usepackage{csquotes}
\usepackage{amsfonts, amssymb, amsmath, amsthm,latexsym,slashed}
\usepackage{mathtools, thm-restate}

\usepackage{bbm}
\usepackage{stmaryrd}
\usepackage{mathrsfs}
\usepackage{bbding}
\usepackage{geometry}
\usepackage{hyperref}
\usepackage{enumitem}
\usepackage{xcolor}
\usepackage[UKenglish]{babel}
\hypersetup{
	colorlinks,
	linkcolor={blue}, 
	citecolor={green!50!black}, 
	urlcolor={cyan!75!blue} 
}

\usepackage[authoryear,square]{natbib} 
\date{\today}
\title{}

\def\<{\left<}
\def\>{\right>}

\def\Z{\mathbb{Z}}
\def\C{{\mathbb{C}}}
\def\R{{\mathbb{R}}}

\def\mat#1#2#3#4{{\begin{pmatrix} #1&#2\\ #3&#4 \end{pmatrix}}}

\newtheorem{df}{Definition}[section]
\newtheorem{lemma}[df]{Lemma}
\newtheorem{thm}[df]{Theorem}
\newtheorem*{thm*}{Theorem}
\newtheorem{prop}[df]{Proposition}
\newtheorem{cor}[df]{Corollary}
\newtheorem{nt}[df]{Notation}

\newtheorem{defprop}[df]{Definition/Proposition}
\theoremstyle{remark}
\newtheorem{rem}[df]{Remark}
\newtheorem{ex}[df]{Example}

\DeclareMathOperator{\Ker}{Ker}

\DeclareMathOperator{\APS}{APS}
\DeclareMathOperator{\ind}{ind}
\newcommand{\gInd}{\ind_{\gamma}}
\DeclareMathOperator{\sfl}{sf}

\DeclareMathOperator{\tr}{tr}

\DeclareMathOperator{\Str}{tr}
\DeclareMathOperator{\grad}{{grad}}

\DeclareMathOperator{\sign}{{sign}}
\DeclareMathOperator{\Isom}{{Isom}}
\DeclareMathOperator{\pr}{pr}
\DeclareMathOperator{\spec}{spec}
\DeclareMathOperator{\Eig}{Eig}
\DeclareMathOperator{\LHS}{LHS}
\DeclareMathOperator{\RHS}{RHS}
\DeclareMathOperator{\even}{{even}}

\DeclareMathOperator{\dist}{dist}

\newcommand{\fT}{\mathsf{T}}

\DeclareMathOperator{\ri}{i}

\newcommand{\ka}{\mathfrak{a}}
\newcommand{\kb}{\mathfrak{b}}
\newcommand{\fc}{\mathsf{c}}

\newcommand{\sE}{\mathscr{E}}
\newcommand{\sM}{\mathscr{M}}

\newcommand{\bbS}{\mathbb{S}}
\newcommand{\sS}{S}
\newcommand{\spinorBun}{\sS M}

\newcommand{\one}{\mathbbm{1}}
\newcommand{\dVol}{\mathrm{d vol}}

\DeclareMathOperator{\End}{End}
\newcommand{\fh}{h}

\DeclareMathOperator{\rT}{\mathrm{T}}

\newcommand{\rd}{\mathrm{d}}

\newcommand{\ft}{\mathsf{t}}

\newcommand{\tangent}{\mathrm{T}}
\newcommand{\coTan}{\mathrm{T}^{*}}

\DeclareMathOperator{\ch}{ch}

\title[Equivariant Index Theorem]{A Lorentzian Equivariant Index Theorem}
\author[O.~Islam]{Onirban Islam}
\address{Institut f\"{u}r Mathematik, Universit\"{a}t Potsdam, Karl-Liebknecht-Stra\ss{}e 24-25, 14476 Potsdam, Germany.}
\email{onirban.islam@math.uni-potsdam.de}
\author[L.~Ronge]{Lennart Ronge}
\address{Institut f\"{u}r Mathematik, Universit\"{a}t Potsdam, Karl-Liebknecht-Stra\ss{}e 24-25, 14476 Potsdam, Germany.}
\email{ronge@uni-potsdam.de}
\subjclass[2010]{}
\keywords{Equivariant Lorentzian index theorem, equivariant spectral flow}
\date{18 February 2026}
\begin{document}

\begin{abstract}
We develop a formula for the equivariant index of a twisted Dirac operator on a compact globally hyperbolic spacetime with timelike boundary on which a group acts isometrically, subject to APS boundary conditions. The formula is the same as in the Riemannian case: the equivariant index for a group element is an integral over the fixed point set of that element plus some boundary terms. The proof uses a surprisingly simple technique for reducing from the equivariant to the non-equivariant regime in order to show an equivariant version of the Lorentzian "index $=$ spectral flow" formula.
\end{abstract}
\maketitle
\tableofcontents 
%
%
%
%
%
%
%
%
%
%
\section{Introduction}

The~\citet{Atiyah_AnnMath_1968_I, Atiyah_AnnMath_1968_III} 
index theorem is the seminal work in geometric analysis that combines geometry, topology, and analysis. 
On an even-dimensional closed Riemannian spin manifold $\Sigma$, it states that the Fredholm index $\ind (\slashed{D})$ of a Dirac operator $\slashed{D}$ is given by the integral of the $\hat{A}$-class of
$\Sigma$~\cite[Theorem (5.3)]{Atiyah_AnnMath_1968_III}. 
There are various generalisations of their work over time, for which we refer to the 
textbook~\citep{Lawson_PUP_1989} 
and the recent 
survey~\citep{Freed_BullAMS_2021}. 
Amongst those, an interesting case occurs when a compact Lie group $\Gamma$ acts on $\Sigma$ such that its action is compatible with all the relevant structures so that it commutes with $\slashed{D}$. 
In this situation, the Fredholm index $\ind (\slashed{D})$ is replaced by the equivariant 
index, a function on $\Gamma$ given by
\[\ind_\gamma(\slashed{D}):=\tr(\gamma|_{\ker(\slashed{D})})-\tr(\gamma|_{\ker(\slashed{D}^*)}).\]
The Lefschetz fixed-point formula due 
to~\citet[Theorem 2.12]{Atiyah_AnnMath_1968_II} 
and~\citet[(5.4)]{Atiyah_AnnMath_1968_III} 
then states that $\gInd (\slashed{D})$ (for any $\gamma \in \Gamma$) is given by the integral of a suitable form over the fixed-point set $\Sigma_{\gamma}$ of $\gamma$. 
The first proof of the Atiyah-Segal-Singer index theorem was essentially topological ($K$-theoretic and cohomological). 
Analytical proofs were developed later 
by~\cite[Theorem 4.15]{Bismut_JFA_1984} 
(probabilitic approach) 
and  
by~\citet[Theorem 3.32]{Berline_BullSocMathFrance_1985} 
(heat kernel method). 

The ellipticity of Dirac operators on a Riemannian (spin) manifold is arguably the quintessential ingredient in the Atiyah-(Segal-)Singer index theorem. 
Since Dirac operators on a Lorentzian (spin) manifold are hyperbolic, a straightforward Lorentzian generalisation of the Atiyah-(Segal-)Singer index theorem was not expected 
until~\citet[theorems 3.3, 3.5, and 4.1]{BSglobal} 
showed that a Dirac operator $D$ on a spatially compact globally hyperbolic spacetime $M$, subject to the Atyiah-Patodi-Singer (APS) boundary condition at the spacelike boundary $\partial M$, is Fredholm. 
They have also proven that $\ind (D)$ is given by the same expression as in 
the~\citet*[Theorem (3.10)]{Atiyah_MathProcCam_I_1975} 
index theorem. 
So far, there are a few 
generalisations~\citep{Braverman_MathZ_2020, vDR, Damaschke, Shen_PureApplAanal_2022, BSlocal} 
of the B\"{a}r-Strohmaier global index theorem. 
Amongst 
these,~\citet{Damaschke} 
has considered a modified index in the presence of a group action, but he studied the $L^2$-index theorem, where the group serves as a way of compensating for the non-compactness of the manifold and the index is still a single number. 
In contrast, we assume compactness but use a refined notion of index. 

The purpose of this article is to put forward the B\"{a}r-Strohmaier global index theorem on an equivariant setting in order to obtain a Lorentzian analogue of the Atiyah-Segal-Singer index theorem. 
But the Lorentzian spin spacetime (see Corollary~\ref{Isplit}) $M$ pertinent to this article has spacelike boundary $\partial M$. 
Hence we have an additional boundary contribution to the Atiyah-Segal-Singer index theorem. 
On a $\Gamma$-equivariant Riemannian spin manifold, such an index theorem was first derived 
by~\citet[Theorem 1.2]{Donnelly_IndianaUMJ_1978}. 
His result can be considered as an equivariant generalisation of the Atiyah-Patodi-Singer index theorem. 
Recent results on this direction can be found, for instance, 
in~\citep{Braverman_JGP_2015, Braverman_IndianaUMJ_2019, Hochs_2020, Hoch_IMRN_2023, Hochs_2024, Sadegh_JFA_2024}.  
Therefore, our result can be perceived as an analogue of the Atiyah-Segal-Singer-Donnelly index theorem in a Lorentzian setting. 
To be precise, our main theorem is as follows. 

%
%
%
\begin{restatable}{thm}{ThmGlobalGInd} 
\label{thm: global_g_ind}
    Let $(M, g)$ be an even-dimensional smooth, compact globally hyperbolic spin spacetime with spacelike boundary $\partial M = \Sigma_{0} \sqcup \Sigma_{1}$ where $\Sigma_{0}$ and $\Sigma_{1}$ are respectively the past and future boundary.  
    Let $E \to M$ be a smooth complex vector bundle over $M$ and $D$ the twisted Dirac operator on the twisted right-handed spinor bundle $S^{+} M \otimes E \to M$. We assume that there is a group $\Gamma$ acting on $M$ and $E$ by isometries that preserve all relevant structures, so $D$ is $\Gamma$-equivariant. 
    Then, for each $\gamma \in \Gamma$, the equivariant index $\gInd (D_{\APS})$ of the twisted Dirac operator $D_{\APS}$ under the Atiyah-Patodi-Singer (APS) boundary condition is given by 
    \begin{align*}
        \gInd (D_{\APS})   
        & = 
        \int_{M_{\gamma}} (2 \pi \ri)^{- \ell} \ri^{-\ell^{\perp}} \ka  
        + \int_{\partial M_{\gamma}} (2 \pi \ri)^{- \ell} \ri^{-\ell^{\perp}} \fT \ka  
        + \kb, 
        \nonumber \\ 
        \ka 
        & := 
        \frac{\hat{A} (M_{\gamma}) \wedge \ch_{\gamma}(E)}{\det^{1/2} \big( \one - \gamma^{\perp} \exp (- R^{\perp}) \big)}, 
        \nonumber \\ 
        \kb 
        & := 
        - 
        \frac{1}{2} \big( \tr (\gamma |_{\ker A (0)}) + \tr (\gamma |_{\ker A (1)}) +  \eta_{\gamma} (A(0)) - \eta_{\gamma} (A(1)) \big). 
    \end{align*}
    Here $M_{\gamma}$ is the fixed-point set of $\gamma$ with the orientation (see Remark~\ref{rem: orientation_fixed_pt_set}) induced by the spin-structure and $\gamma$ on each connected component $M_{\gamma}^{n}$ of $M_{\gamma}$, $\ell := \dim M_{\gamma}^{n} / 2$ and $\ell^{\perp} := \dim (M_{\gamma}^{\perp} |_{ M_{\gamma}^{n}}) / 2$ are half of the dimension and codimension of $M_{\gamma}^{n}$, $\hat{A} (M_{\gamma})$ is the $\hat{A}$-form (see~\eqref{eq: def_A_hat_genus_fixed_pt_set}) of $M_{\gamma}$, $\ch_{\gamma} (\nabla^E)$ is the localised Chern character (see~\eqref{eq: def_Chern_char_form_twist_bundle}) of $E$, $\gamma^{\perp}$ is the restriction of $\rd \gamma$ to the normal bundle $M_{\gamma}^{\perp}$ of $M_{\gamma}$, $R^{\perp}$ is the curvature of $M_{\gamma}^{\perp}$, $\fT \ka$ is the transgression form (see Lemma~\ref{lem: transgression_form_ASS}) of $\ka$, and $A (i), i = 0, 1$ is the twisted Dirac operator on $\Sigma_{i}$. 

    If, in addition, $M_{\gamma}$ has a spin-structure then 
    \begin{equation*}
        \ka = \ri^{-\ell^{\perp}} \frac{\hat{A} (M_{\gamma}) \wedge \ch_{\gamma} (\nabla^E)}{\ch_{\gamma} (\sS M_{\gamma}^{\perp})},  
    \end{equation*}
    where $\sS M_{\gamma}^{\perp}$ is the spinor bundle over $M_{\gamma}^{\perp}$. 
\end{restatable}
%

%
%

Our proof relies heavily on the existing work on both the Lorentzian index theorem and the Riemannian equivariant index theorem. 
The general proof strategy follows that of the classical Lorentzian index theorem 
by~\citet{BSglobal} 
and we also use some of their non-equivariant results. 
At the point where they have used 
the~\citet*{Atiyah_MathProcCam_I_1975} 
index theorem, we employ the equivariant one by combining the result (see Theorem~\ref{thm: BGVD_ind_thm}) 
of~\citet[Theorem 3.32]{Berline_BullSocMathFrance_1985} 
(see also~\citep[Theorem 6.16]{Berline_Springer_2004})
with that of~\citet[Theorem 1.2]{Donnelly_IndianaUMJ_1978}.  
In a sense, the proof of \cite{BSglobal} is a machine for turning Riemannian index theorems into Lorentzian index theorems. 
We, in turn, build machinery to transform non-equivariant results into equivariant ones. Combining both, we obtain a way of turning Riemannian equivariant index theorems into Lorentzian ones. The trick for reducing equivariant quantities to non-equivariant ones is surprisingly simple and versatile, but in order to be able to apply it, some preparatory work must be done. 
The results of each of the preparatory sections are of sufficient generality to be interesting in their own right.

In Section~\ref{sec: spacetime_equivariant_splitting}, we develop a suitable splitting to work in for the rest of the paper. We show that a compact globally hyperbolic spacetime $M$ with spacelike boundary $\partial M$ can be decomposed as a product 
\[M\cong [0,1]\times\Sigma\]
in such a way that every isometry of $M$ is given by an isometry of $\Sigma$ acting in the second component (see Corollary \ref{Isplit}). In particular, this means that our group action preserves each hypersurface $\{t\}\times\Sigma$ and we get a well-defined family $(A(t))_{t\in[0,1]}$ of hypersurface Dirac operators that are also equivariant under the group action. This splitting also allows us to define a Riemannian equivariant metric on $M$.

In Section~\ref{sec: G_ind_G_sf}, we develop our machinery for reducing equivariant quantities to non-equivariant ones. We do this by decomposing into eigenspaces of the action of a group element $\gamma$. On each eigenspace, $\gamma$ acts as a multiple of the identity and so the $\gamma$-index and $\gamma$-spectral flow are just multiples of their non-equivariant counterparts. Summing over all eigenspaces, we obtain the formulas (see theorems \ref{inddeco} and \ref{sfdeco}) 
\[\ind_\gamma(D)=\sum\limits_{\lambda\in \Eig(\gamma)}\lambda\ind(D|_{E_\lambda(\gamma)})\]
and
\[\sfl_\gamma(A)=\sum\limits_{\lambda\in \Eig(\gamma)}\lambda \sfl(A|_{E_\lambda(\gamma)}), \]
where $E_\lambda(\gamma)$ denotes the $\lambda$-eigenspace of $\gamma$. These provide a simple way of proving theorems about equivariant indices and spectral flows from corresponding non-equivariant theorems, as long as everything restricts well to closed subspaces.

In our geometric setting described in Section~\ref{sec: set-up}, this is not so simple: the restriction of a Dirac operator to an eigenspace of the group action is not a Dirac operator again. In order to get something that restricts well to eigenspaces, we show in Section~\ref{sec: relation_abstract_model_op} that the Dirac operator is unitarily equivalent to the operator
\[\mathsf{c}(\nu(0))N^{-\frac{1}{2}}(\partial_t-{\ri}B)N^{-\frac{1}{2}},\]
where $\nu(0)$ is the unit normal to the past boundary, $N$ is the lapse function and $B$ is a family of self-adjoint operators directly related to the hypersurface Dirac operators $A$ (see Theorem \ref{Diractoabstract}). In the Riemannian case, we get the same result without $-{\ri}$. This generalises \cite[Proposition 3.5]{vDsplit}, where a similar result is shown under the additional assumption that the unit normal field $\nu$ is parallel. This reduces the Dirac index to that of the abstract operator $\partial_t-{\ri}B$, which then allows us to apply the results of Section~\ref{sec: G_ind_G_sf}. However, this result is also of independent interest as a means of relating the geometric setting to that of abstract operator families.

In Section~\ref{sec: g_ind_thm}, we put everything together. Combining the results of sections~\ref{sec: G_ind_G_sf} and~\ref{sec: relation_abstract_model_op} with suitable theorems from the non-equivariant setting, we obtain that the equivariant index of both the Lorentzian Dirac operator $D$ and an auxiliary Riemannian Dirac operator $\hat D$ coincide with the equivariant spectral flow of the family $A$ of hypersurface Dirac operators. 

We then describe the relevant equivariant characteristic forms and compute the transgression form $\fT\ka$. This allows us to conclude that the right hand side of Theorem~\ref{thm: global_g_ind} is not only well defined, but also coincides with the right hand side of the Riemannian equivariant index theorem \ref{thm: BGVD_ind_thm}, which is obtained by combining the index theorems due to 
\cite{Berline_BullSocMathFrance_1985} 
and~\cite{Donnelly_IndianaUMJ_1978}. Putting everything together, we obtain
\[\ind_\gamma(D)=\sfl_\gamma(A)=\ind_\gamma(\hat D)=\RHS,\]
where $\RHS$ denotes the right hand side of either index theorem.

%
%
%
%
%
%
%
%
%
%
\section{Spacetime splitting} 
\label{sec: spacetime_equivariant_splitting}
The proof of the global index theorem works with eigenspaces of Dirac operators on Cauchy slices. To do this in the equivariant case, we need to make sure that the group action is well defined on these eigenspaces. This requires two things: that the action preserves the slices and that it commutes with the slice operators. In order to avoid problems with boundary conditions, we furthermore want to make sure that the lapse function is constantly one near the boundary. We will see that this is always possible. In fact, somewhat surprisingly, we can always choose a spacetime splitting in such a way that the group action comes from an action on $\Sigma$ that is constant in time. 
To begin with, we recall that a hypersurface $\Sigma$ in a spacetime $\sM$ is called a Cauchy hypersurface if it is intersected exactly once by each timelike curve. A spacetime $\sM$ is called globally hyperbolic if it admits a Cauchy hypersurface.

\begin{df} \label{def: cpt_GHST_sBdy}
By a (timelike compact) globally hyperbolic manifold $M$ with spacelike boundary $\partial M$, we mean the region $M$ between two disjoint spacelike Cauchy hypersurfaces $\Sigma_{0}, \Sigma_{1}$ in a globally hyperbolic manifold without boundary. 
\end{df}
\begin{nt}
    For the rest of this section (and also later on) we will let $M$ denote a globally hyperbolic manifold with spacelike boundary $\partial M = \Sigma_{0} \sqcup \Sigma_{1}$ and metric $g$. 
\end{nt}
Every such manifold is diffeomorphic to a product
\begin{equation*} 
    M \cong [0, 1] \times \Sigma. 
\end{equation*}
Note however, that this diffeomorphism is generally not an isometry, i.e. the metric is generally not a product metric.

To ensure good behaviour near the boundary, we shall need the following lemma.
\begin{lemma}
\label{lapsef}
    There is a smooth time function $f$ with past directed gradient such that near the past boundary $\Sigma_0$, we have $f(x)= \dist(x,\Sigma_0)$ and near the future boundary $\Sigma_1$, we have $f(x)=1-\dist(x,\Sigma_1)$. Here $\dist$ refers to the Lorentzian distance, i.e. the length of a longest causal curve.
\end{lemma}
\begin{proof}
    Let $h_1$ be any smooth time function on $M$ with timelike gradient, $h_1(\Sigma_0)=\{0\}$ and $h_1(\Sigma_1)=\{1\}$ (e.g. project onto the time component in a smooth splitting). Define $h_0(x):=\dist(x,\Sigma_0)$ and $h_2:=1-\dist(x,\Sigma_1)$. Choose $\chi_0,\chi_1,\chi_2\colon [0,1]\rightarrow [0,1]$ smooth and $0<\epsilon<\frac{1}{6}$ such that
    \begin{itemize}
        \item While $h_1(t)<2\epsilon$, $h_0(t)$ is smooth with timelike gradient.
        \item While $h_1(t)>1-2\epsilon$, $h_2(t)$ is smooth with timelike gradient.
        \item All $\chi_i'\geq 0$.
        \item $\chi_0$ is the identity on $[0,\epsilon]$, has positive derivative on $[0,2\epsilon]$ and is constantly $\frac{1}{3}$ around $[3\epsilon,1]$.
        \item $\chi_2(t)=t-\frac{2}{3}$ for $t\geq1-\epsilon$, moreover $\chi_2$ has positive derivative on $[1-2\epsilon,1]$ and vanishes around $[0,1-3\epsilon]$.
        \item $\chi_1$ has positive derivative on $[2\epsilon,1-2\epsilon]$, vanishes around $[0,\epsilon]$ and is constantly $\frac{1}{3}$ around $[1-\epsilon,1]$.
    \end{itemize}
    Note that at any $t\in[0,1]$, some $\chi_i'(t)$ is strictly positive.

    Define
    \[f:=\chi_0\circ h_0 +\chi_1\circ h_1+\chi_2\circ h_2.\]
    As every summand has past timelike or zero gradient and at any point at least one summand has non-zero gradient, $f$ has timelike gradient. Moreover, on $h_1^{-1}([0,\epsilon])$, we have
    \[f=\chi_0\circ h_0+0=\dist(\cdot,\Sigma_0)\]
    and on $h_1^{-1}([1-\epsilon,1])$, we have
    \[f(x)=\frac{1}{3}+\frac{1}{3}+\chi_2(1-\dist(x,\Sigma_1))=1-\dist(x,\Sigma_1).\]
\end{proof}
\begin{prop}
\label{Gsplit}
        Let $\Gamma$ be a compact Hausdorff topological group acting by timeorientation preserving isometries on $M$.
        Then there is a $\Gamma$-manifold $\Sigma$ and a diffeomorphism
        \[\phi \colon [0,1]\times \Sigma\rightarrow M \]
        such that for any $\gamma\in \Gamma$, $x'\in \Sigma$, $t\in [0,1]$, we have
        \[\phi(t,\gamma.x')=\gamma\phi(t,x').\]
        This can be chosen such that every $\Sigma_t:=\phi(\{t\}\times\Sigma)$ is a Cauchy hypersurface and the pull-back of the metric takes the form
        \[\phi^*g=-N^2dt^2+g_t\]
        for a family $g_t$ of Riemannian metrics on $M$ and a smooth function $N$ that is constantly $1$ near the boundary.     
    \end{prop}
    \begin{rem}
        The key step of the proof, obtaining an equivariant time function, was already done in \cite[Theorem 1]{Mue}. However, we repeat it here as it is fairly short and elegant.
    \end{rem}
    \begin{proof}
        Choose a time function $f$ as in Lemma \ref{lapsef}. Define 
        \[f_\Gamma(x):=\int\limits_\Gamma f(\gamma^{-1}(x))dH(\gamma),\]
        where $H$ denotes the normalised left Haar measure of $\Gamma$. 
        
        As the measure is translation invariant, we have
        \[f_\Gamma(hx)=\int\limits_\Gamma f((h^{-1}\gamma)^{-1}(x))dH(\gamma)=\int\limits_\Gamma f(\gamma^{-1}(x))dH(\gamma)=f_\Gamma(x).\]
        Thus all slices $f_{\Gamma}^{-1}(t)$ are $\Gamma$ invariant.
        
        As $\grad f$ is past timelike and $d\gamma$ fixes the positive timecone, $\gamma_*\grad(f)$ is always past timelike. Thus
        \begin{align*}
            \grad(f_\Gamma)(x) = \int\limits _\Gamma\grad(f\circ{\gamma^{-1}})(x)dH(\gamma)=\int\limits_\Gamma \gamma_*\grad(f)(x)dH(\gamma)
        \end{align*}
        is past timelike, so $f_\Gamma$ is again a Cauchy time function with timelike gradient.
        
        As $f$ (and the metric)is $\Gamma$-equivariant, so is its $\grad f$ and thus also the vector
        \[V:=\frac{\grad(f_\Gamma)}{\|\grad{f_\Gamma}\|^2}\]
        (for Lorentzian "norm squared").

        Let $\phi_t$ denote the flow induced by $V$, i.e. $\partial_t\phi_t=V\circ\phi_t$, $\phi_0=id$. 
        As  $V$ is equivariant under the group action, so is $\phi$. Moreover,
        \[\partial_t f_\Gamma(\phi_t(x))=\<\grad(f_\Gamma),V \>=1,\]
        so $\phi_t$ maps $f_\Gamma^{-1}(s)$ to $f_\Gamma^{-1}(s+t)$ for any $s\in t$.
        Define $\Sigma_t:=f_\Gamma^{-1}(t)$, $\Sigma:=\Sigma_0$.
        As flow curves can only "enter" or "leave" M in (i.e. stay well defined until hitting) $\Sigma_0$ resp. $\Sigma_1$, we get a well-defined diffeomorphism
        \begin{align*}
            \phi\colon [0,1]\times \Sigma&\to M\\
            (t,x')&\mapsto \phi_t(x')
        \end{align*}
        with inverse
        \begin{align*}
            \psi\colon M&\to [0,1]\times \Sigma\\
            x&\mapsto(\phi_{-f_\Gamma(x')}(x'),f_\Gamma(x')).
        \end{align*}
        As $\phi_t$ commutes with the group action, we have for any $\gamma \in \Gamma$
        \[\gamma \phi(t,x')=\phi(t, \gamma x').\]
        For any inextendible causal curve $\gamma$, $f_\Gamma\circ\gamma$ is a strictly increasing bijection of $[0,1]$ to itself, so $\gamma$ must intersect each $\Sigma_t$ exactly once. Moreover, as the gradient of $f_\Gamma$ is perpendicular to its level sets, each flow curve intersects each $\Sigma_t$ orthogonally. This means that the metric has no off-diagonal terms, i.e. takes the form 
        \[\phi^*g=-Ndt^2+g_t.\]
        It remains to show that $N$ is 1 near the boundary. Here we use that isometries must preserve the distance from the boundary. Near $\Sigma_0$, we have
        \begin{align*}
            f_\Gamma(x)&=\int\limits_\Gamma h^*f(x)dH(h)\\
            &=\int\limits_\Gamma h^*\dist(x,\Sigma_0)dH(h)\\
            &=\int\limits_\Gamma \dist(x,\Sigma_0)dH(h)\\
            &=\dist(x,\Sigma_0).
        \end{align*}
        In particular, the gradient of $f_\Gamma$ has unit length near $\Sigma_0$. In the same way, we can show that it has unit length near $\Sigma_1$. 
        Thus
        \begin{align*}
            N(x)^2&=-\phi^*\gamma(\partial_t,\partial_t)(x)\\
            &=-\gamma(\partial_t\phi_t(x),\partial_t\phi_t(x))\\
            &=-\gamma(V(x),V(x))\\
            &=-\frac{1}{\|\grad(f_\Gamma)(x)\|^2}\\
            &=1,
        \end{align*}
        which shows that $N$ is $1$ near the boundary.
    \end{proof}
    The above proof also works for spacetimes without boundary. The key assumption is the compactness of $\Gamma$, which rules out time translations. In our setting, however, we will have a compact manifold (which requires a boundary in the globally hyperbolic case) and this already forces the action to factor through a compact group.
    \begin{prop}
    \label{Icomp}
        If $M$ is compact, the group of isometries on $M$ is compact (with the compact-open topology).
    \end{prop}
    \begin{proof}
    We can assume without loss of generality that $M$ is connected, otherwise isometries are given by an isometry on each component. It also suffices to show this for timeorientation preserving isometries, as composition with any fixed time-reversing isometry gives a homeomorphism between timeorientation preserving and reversing ones.

    We use that the group of isometries of the past boundary $\Sigma_0$ is compact (\cite[Theorem I']{DvW}) (with respect to the topology induced by the distance, which coincides with the compact-open topology). We will show that an isometry is uniquely determined by its restriction to the past boundary $\Sigma_0$ and that the extension process is continuous. For metrizing the topology of $\Isom(M)$, we choose an arbitrary Riemannian metric on $M$ with respect to which we measure the distance of two isometries (the induced topology will again be the compact-open topology). For everything else, we still use the Lorentzian metric.

    Let $\phi$ be an isometry and $x\in M$. There must be a longest causal curve from $x$ to $\Sigma_0$ that is a geodesic (this follows e.g. from \cite[Proposition 14.19 and Lemma 14.21]{ONE}). This must hit $\Sigma_0$ orthogonally, otherwise there would be a longer curve. Let $y$ be the point where this geodesic hits $\Sigma_0$. Then we can write
    \[x=\exp_y(\lambda\nu(y)),\]
    where $\nu(y)$ denotes the past unit normal to $\Sigma_0$ at $y$ and $\lambda\in \R$. As $\phi$ is a timeorientation preserving isometry, it must map $\Sigma_0$ to $\Sigma_0$ and past normal vectors of $\Sigma_0$ to past normal vectors of $\Sigma_0$. Moreover, it preserves geodesics and the exponential function.
    Thus we have
    \[\phi(x)=\exp_{\phi(y)}(\lambda\nu(\phi(y))).\]
    Now let instead $\phi_n$ be a sequence of isometries. As $\Isom(\Sigma_0)$ is compact, the sequence $\phi_n|_{\Sigma_0}$ has a subsequence converging uniformly to a limit $\phi_0$. Without loss of generality, we assume that the whole sequence converges. Define (choosing values of $y$ and $\lambda$ as above for any value of $x$)
    \[\phi(x):=\exp_{\phi_0(y)}(\lambda\nu(\phi_0(y))).\]
    As the map $z\mapsto \exp_z(\lambda\nu(z))$ is smooth and thus in particular Lipschitz, we obtain that
    \[\phi_n(x)=\exp_{\phi_n(y)}(\lambda\nu(\phi_n(y)))\]
    converges uniformly to $\phi(x)$.
    It remains to show that $\phi$ is an isometry. The differential of $\phi_n$ at $x$ is (using again that $\phi$ preserves the exponential map and assuming that $v$ is small enough so the following is defined)
    \[d_x\phi_n(v)=\exp_{\phi_n(x)}^{-1}\phi_n(\exp_x(v)).\]
    The right hand side converges uniformly, so the $\phi_n$ converge in $C^1$. This shows that the limit $\phi$ must also be $C^1$ with \[d\phi=\lim\limits_{n\rightarrow\infty}d\phi_n.\]
    As all $d_x\phi_n$ are Lorentz isometries, so is $d_x\phi$ for any $x\in M$, i.e. $\phi\in\Isom(M)$.
    Thus we have shown sequential compactness, which, for metric spaces, is equivalent to compactness.
    \end{proof}
    Combining this with what we have shown before, we obtain
    \begin{cor}
    \label{Isplit}
    Every compact globally hyperbolic manifold with spacelike boundary is isometric to a manifold $[0,1]\times \Sigma$ with metric 
    \[-N^2dt^2+g_t\]
    such that $N$ is $1$ near the boundary and every time-oriented isometry of $[0,1]\times \Sigma$ to itself is constant in time, i.e. of the form $\phi(t,x)=(t,\tilde\phi(x))$ for some $\tilde \phi$ on $\Sigma$.
    \end{cor}
    \begin{proof}
        Apply Proposition \ref{Gsplit} with $\Gamma$ the group of time-oriented isometries. Proposition \ref{Icomp} ensures that this choice of $\Gamma$ is admissible. As the metric on the product, take the pull-back under the constructed diffeomorphism, turning it into an isometry.
    \end{proof}
\section{Group index and group spectral flow}
\label{sec: G_ind_G_sf}
\subsection{Spectral subspaces and equivariant operators}
Our approach to dealing with the equivariant index and spectral flow will rely heavily on spectral decompositions. 
\begin{df}
    For any Borel set $I\subset\C$, and any normal operator $S$, let $P_I(S)$ denote the corresponding spectral projection, i.e. the characteristic function of $I$ applied to $S$ via functional calculus. Let $E_I(S)$ denote the range of $P_I(S)$. For $I=\{\lambda\}$, we write $E_\lambda(S)$ and $P_\lambda(S)$ instead. We write $\Eig(S)$ for the set of eigenvalues of $S$.
    Let $S_{X\to Y}$ denote the restriction of $S$ to a subspace $X$, with codomain restricted to a subspace $Y$.
\end{df}
\begin{nt}
    Throughout this section, we consider the following setting: Let $\Gamma$ be a group. Let $H_1$ and $H_2$ be Hilbert spaces with a unitary $\Gamma$-action. Let $\gamma\in \Gamma$ and let $\gamma_1$ and $\gamma_2$ denote the action of $\gamma$ on $H_1$ resp. $H_2$. Let $T\colon H_1\to H_2$ be a (possibly unbounded) operator that is $\gamma$-equivariant, i.e. $\gamma_1$ and $\gamma_1^{-1}$ fix the domain of $T$ and $T\gamma_1=\gamma_2T$.
\end{nt}
In our case, the relevant spectral subspaces will be finite-dimensional and thus direct sums of eigenspaces.

If two operators commute, they preserve each others spectral subspaces. For maps between different spaces, there is the following analogue. 
Note that this holds symmetrically for $T^*$:
\begin{lemma}
    \label{interT}
    \begin{enumerate}
        \item $T$ preserves spectral subspaces, i.e. it maps $E_I(\gamma_1)$ to $E_I(\gamma_2)$ for any Borel set $I$.
        \item $\gamma_1$ preserves $\Ker(T)$.
    \end{enumerate}
\end{lemma}
\begin{proof}
    Observe that $\hat T:=\mat{0}{0}{T}{0}$ commutes with $\hat\gamma:=\mat{\gamma_1}{0}{0}{\gamma_2}$. The claims then follow from the fact that commuting operators preserve each others spectral subspaces.
\end{proof}
In particular, for $I\subseteq \Eig(\gamma_1)$, we have a direct sum decomposition
\[T=\bigoplus\limits_{\lambda\in I}T|_{E_\lambda(\gamma_1)}\oplus T|_{E_{\C\backslash I}(\gamma_1)}.\]
The following Lemma tells us that everything "interesting" will be happening in finitely many eigenspaces of the $\gamma$-action, i.e. we can choose the above decomposition such that nothing happens in the last summand.
\begin{lemma}
    \label{finiteES}
    If $\Ker(T)$ is finite dimensional, then $\Ker(T)$ is contained in the sum of finitely many eigenspaces of $\gamma_1$ and we have
    \[\tr(\gamma_1|_{\Ker(T)})=\sum\limits_{\lambda\in \Eig(\gamma_1)}\lambda\dim(\Ker(T|_{E_\lambda(\gamma_1)})).\]

    If $H_1=H_2$, $T$ is self-adjoint and  $E_I(T)$ is finite dimensional, then $E_I(T)$ is contained in the sum of finitely many eigenspaces of $\gamma_1$ and we have
    \[\tr(\gamma_1|_{E_I(T)})=\sum\limits_{\lambda\in \Eig(\gamma_1)}\lambda\dim(E_I(T|_{E_\lambda(\gamma_1)})).\]
\end{lemma}
\begin{proof}
    Let $X$ denote either $\Ker(T)$ or $E_I(T)$. As $X$ is invariant under $\gamma_1$, $\gamma_1|_X$ is a unitary map on a finite-dimensional space and thus has a basis of eigenvectors. As these are also eigenvectors of the unrestricted $\gamma_1$ (and only finitely many), the first part of the claim follows.

    Writing the trace as the sum of the eigenvalues, we get
    \begin{align*}
        \tr(\gamma_1|_{X})&=\sum\limits_{\lambda\in \Eig(\gamma_1|_{X})}\lambda \dim(E_\lambda(\gamma_1|_{X}))\\
        &=\sum\limits_{\lambda\in \Eig(\gamma_1)}\lambda \dim(E_\lambda(\gamma_1)\cap X).\\
    \end{align*}
    As taking kernels or spectral subspaces is compatible with restriction to closed invariant subspaces, this concludes the second part of the claims.
\end{proof}

\subsection{Equivariant index}
\label{sec: g_ind}
There are various equivalent ways of defining a group index for a Fredholm operator that commutes with the action of a group $\Gamma$. The properties of the operator ensure that its kernel and cokernel are finite-dimensional vector spaces preserved by the group action. They can either be viewed as representations of the group (which represent an element of the representation ring) or as modules over the group $C^*$-algebra (which represent an element in the K-theory group $K_0(C^*(\Gamma))$). In either case, one ends up in a setting where one can take a formal difference of the kernel and cokernel. The equivariant index of $T$ is then defined 
as~\cite[p. 519]{Atiyah_AnnMath_1968_I}
\[\ind_\Gamma(T)=[\ker(T)]-[\ker(T^*)].\]
The approach that we will take is to look, rather than at the representations themselves, at the characters associated to them.
The relation between the two is the follows: To a finite-dimensional unitary representation $\pi$ of a group $\Gamma$, on can asociate the character, i.e. the function 
\begin{align*}
    \chi_\pi\colon \Gamma&\to\C\\
    \gamma&\mapsto \tr(\pi(\gamma)).
\end{align*}
This association is invariant under isomorphism and additive under direct sums, so it descends to a map from the representation ring by setting
\[\chi_{[\pi]-[\rho]}(\gamma)=\tr(\pi(\gamma))-\tr(\rho(\gamma)).\]
Moreover, a representation is uniquely (up to unitary equivalence) defined by its character, so this induced map is injective. Thus we can define the index as a function from $\Gamma$ to $\C$, and this will contain the same information as the corresponding element in the representation ring. This definition has the advantage that one does not need to look at the group as a whole, but only has to consider individual elements of $\Gamma$. We thus refrain from venturing into representation theory or K-theory any further and simply use the following definition.

\begin{df} \label{def: g_ind}
    If $T$ is a $\gamma$-equivariant Fredholm operator then one defines 
    \[\ind_\gamma(T):=\tr(\gamma_1|_{\Ker(T)})-\tr(\gamma_2|_{\Ker(T^*)}).\]
\end{df}
Note that in case $\gamma=id$, this coincides with the usual Fredholm index. The group index is thus an extension of the Fredholm index. In order to be able to transfer results from the usual index, we use an eigenspace decomposition for $\gamma$ and exploit that on each eigenspace, the $\gamma$-index is just a multiple of the usual index.

\begin{thm}
\label{inddeco}
        Let $T\colon H_1\rightarrow H_2$ be an equivariant Fredholm operator between Hilbert spaces with an isometric group action. Then 
        \[\ind_\gamma(T)=\sum\limits_{\lambda\in \Eig(\gamma_1)}\lambda\ind(T|_{E_\lambda(\gamma_1)\rightarrow E_\lambda(\gamma_2)}).\]
\end{thm}
\begin{proof}
    By Lemma \ref{finiteES}, we have
    \begin{align*}
        \tr(\gamma_1|_{\Ker(T)})&=\sum\limits_{\lambda\in \Eig(\gamma_1)}\lambda \dim(\Ker(T|_{E_\lambda(\gamma_1)\to E_\lambda(\gamma_2)}).\\
    \end{align*}
    Doing the same for $T^*$, we obtain
    \begin{align*}
        \tr(\gamma_2|_{\Ker(T^*)})&=\sum\limits_{\lambda\in \Eig(\gamma_1)}\lambda \dim(\Ker(T^*|_{E_\lambda(\gamma_1)\to E_\lambda(\gamma_2)})\\
        &=\sum\limits_{\lambda\in \Eig(\gamma_1)}\lambda \dim(\Ker((T|_{E_\lambda(\gamma_1)\to E_\lambda(\gamma_2)})^*).\\
    \end{align*}
    As the index is the difference of kernel and cokernel, we obtain as the difference of these two equations
    \[\ind_\gamma(T)=\sum\limits_{\lambda\in \Eig(\gamma_1)}\lambda\ind(T|_{E_\lambda(\gamma_1)\rightarrow E_\lambda(\gamma_2)}).\]
\end{proof}
This characterization of the equivariant index immediately yields that properties like homotopy invariance carry over from the classical to the equivariant index.
\subsection{Equivariant spectral flow}
The definition of equivariant spectral flow follows the same pattern as that of the equivariant index. We simply replace dimensions by traces of $\gamma$ over the corresponding space. As for the index, one could equivalently consider classes in the representation ring or in K-theory.

Consider a family $(A(t))_{t\in[0,1]}$ of unbounded self-adjoint Fredholm operators in a Hilbert space $H$ that is continuous (in $t$) with respect to the gap metric (equivalently, $(A(t)+i)^{-1}$ depends continuously on $t$, see \cite[Theorem 1.1]{BLP}). The gap metric induces a fairly coarse topology that is more general than what we are going to need. For our purpose, the following is sufficient: If all $A$ have the same domain $W$ and are norm continuous as a family of bounded operators $W\to H$, where $W$ carries the graph norm of any $A$, then $A$ is also continuous with respect to the gap metric.

Let us start by recalling the definition of ordinary spectral flow (see \cite[Definition 2.12 and surrounding propositions]{BLP})
\begin{defprop}
    There is a partition $0=t_0<t_1<\dots<t_n=1$ of $[0,1]$ and numbers $a_1,\dots,a_n\in\R$ such that $P_{[-a,a]}(A(t))$ has finite rank and is continuous in $t$ for $t\in[t_{n-1},t_n]$. We shall call a choice of such $t_k$ and $a_k$ a flow partition for the family $A$.
    For such a flow partition, the spectral flow of the family $A$ is defined as
    \[\sfl(A)=\sum\limits_{k=1}^n\dim(E_{[0,a_k]}(A(t_k)))-\dim(E_{[0,a_k]}(A(t_{k-1}))).\]
    The spectral flow is independent of the choice of flow partition.
\end{defprop}
We now turn to the equivariant setting.
\begin{nt}
    For the rest of this section, let $A(t)$ be a gap continuous family of self-adjoint  Fredholm operators commuting with a unitary operator $\gamma$.
\end{nt}
The operator $\gamma$ should, of course, be thought of as the action of a group element, but the choice of group plays no role in our considerations. In fact, we only need that $\gamma$ is normal and commutes with $A(t)$.
\begin{df}
    For any flow partition of $A$, we define the $\gamma$ spectral flow of $A$ as
    \[\sfl_\gamma(A)=\sum\limits_{k=1}^n\tr(\gamma|_{E_{[0,a_k]}(A(t_k))})-\tr(\gamma|_{E_{[0,a_k]}}(A(t_{k-1}))).\]
\end{df}
For $\gamma=1$, one obtains ordinary spectral flow. Well-definedness of this can be shown completely analogous to that of the ordinary spectral flow. However, it also follows automatically from the following result, which is at the same time exactly what we need to relate the equivariant spectral flow to the non-equivariant one.
\begin{thm}
\label{sfdeco}
        For any flow partition, we have
        \[\sfl_\gamma(A)=\sum\limits_{\lambda\in \Eig(\gamma)}\lambda \sfl(A|_{E_\lambda(\gamma)}).\]
        In particular, $\sfl_\gamma(A)$ is well defined, i.e. independent of the choice of flow partition.
\end{thm}
    \begin{proof}
        Let $(t_k)_{0\leq k\leq n}$ and $(a_k)_{1\leq k\leq n}$ be a flow partition for $A$. As the spectral projections of restrictions of $A$ are finite rank and continuous if those of $A$ are, this is also a flow partition for any restriction of $A$ to an invariant subspace.  Using Lemma \ref{finiteES}, we compute
        \begin{align*}
            \sfl_\gamma(A)&=\sum\limits_{k=1}^n\tr(\gamma|_{E_{[0,a_k]}(A(t_k))})-\tr(\gamma|_{E_{[0,a_k]}(A(t_{k-1}))})\\
            &=\sum\limits_{k=1}^n\sum\limits_{\lambda\in \Eig(\gamma)}\lambda\dim(E_{[0,a_k]}(A(t_k)|_{E_\lambda(\gamma)}))-\sum\limits_{\lambda\in \Eig(\gamma)}\lambda\dim(E_{[0,a_k]}(A(t_{k-1})|_{E_\lambda(\gamma)}))\\
            &=\sum\limits_{\lambda\in \Eig(\gamma)}\lambda\sum\limits_{k=1}^n\dim(E_{[0,a_k]}(A(t_k)|_{E_\lambda(\gamma)}))-\dim(E_{[0,a_k]}(A(t_{k-1})|_{E_\lambda(\gamma)}))\\
            &=\sum\limits_{\lambda\in \Eig(\gamma)}\lambda\sfl(A|_{E_\lambda(\gamma)}).\\
        \end{align*}
    \end{proof}

    This can be used to easily translate results on ordinary spectral flow into results on equivariant spectral flow. To illustrate this, we give a very simple alternative proof for a list of known properties (c.f. \citep{IJW, GSFunique}) of the equivariant spectral flow, using their non-equivariant analogues.
    
    \begin{thm}
    \label{sfaxioms}
        The $\gamma$ spectral flow has the following properties:
        \begin{enumerate}
            \item If all $A(t)$ are invertible, then $\sfl_\gamma(A)=0$.
            \item The spectral flow of a direct sum is given by
            \[\sfl_\gamma(A\oplus B)=\sfl_\gamma(A)+\sfl_\gamma(B).\]
            \item For a continuous homotopy $(A(t,s))$ with fixed endpoints we have
            \[\sfl_\gamma(A(\cdot,1))=\sfl_\gamma(A(\cdot,0)).\]
            \item If $H$ is finite dimensional, then 
            \[\sfl_\gamma(A)=\tr(\gamma|_{E_{(-\infty,0)}(A(0))})-\tr(\gamma|_{E_{(-\infty,0)}(A(1))})\]
            \item The spectral flow of a concatenation of paths is given by
            \[\sfl_\gamma(A*B)=\sfl_\gamma(A)+\sfl_\gamma(B).\]
        \end{enumerate}
    \end{thm}
\begin{proof}
    Each of the Axioms can be written in the form
    \[\LHS_\gamma=\RHS_\gamma.\]
    As the classical spectral flow satisfies all the axioms, we have for every $\lambda\in \Eig(\gamma)$
    \[\LHS_1|_{E_\lambda(\gamma)}=\RHS_1|_{E_\lambda(\gamma)}\]
    where the restriction is to be interpreted as restricting every operator inside the expression.
    As all constructions involved are compatible with restrictions to eigenspaces (e.g. the restriction of a homotopy is again a homotopy) and traces are additive over direct sums (for 4.), we get using Theorem \ref{sfdeco}
    \begin{align*}
        \LHS_\gamma&=\sum\limits_{\lambda\in\Eig(\gamma)}\lambda \LHS_1|_{E_\lambda(\gamma)}\\
        &=\sum\limits_{\lambda\in\Eig(\gamma)}\lambda \RHS_1|_{E_\lambda(\gamma)}\\
        &=\RHS_\gamma.
    \end{align*}
\end{proof}
\begin{rem} The equivariant spectral flow is uniquely determined by these properties (or even a suitable subset thereof), as shown in \cite[Theorem 2.1]{GSFunique}. A short alternative proof of this could be achieved using the methods above, given a suitable version of the uniqueness theorem for non-equivariant spectral flow. However, as the existing results use slightly different choices of axioms from \citep{GSFunique} and are stated for maps into $\Z$ (while we would require, a priori, maps into $\C$), we will not go into this here.
\end{rem}

%
%
%
%
%
%
%
%
%
%
\section{Set-up: The Dirac bundle} 
\label{sec: set-up}

Let $M$ be an even-dimensional timelike compact globally hyperbolic smooth spin-manifold with spacelike boundary $\partial M$ (Definition~\ref{def: cpt_GHST_sBdy}). 
Let $\Gamma$ be a group acting by isometries on $M$ that preserve spin-structure and time-orientation.

By Corollary~\ref{Isplit}, we may assume without loss of generality that $M$ is a smooth product manifold $[0, 1] \times \Sigma$ with boundary $\partial M := \Sigma_{0} \sqcup \Sigma_{1}$, where $\Sigma_t : =\{t\} \times \Sigma$ is a smooth spacelike Cauchy hypersurface and $\Sigma_{0}$ is in the causal past of $\Sigma_{1}$.  
The spacetime metric on $M$ is given by
\begin{equation} \label{eq: spacetime_metric}
    g=-N^2dt^2+g_t, 
\end{equation}
where $N$ is a smooth positive real-valued function on $[0, 1] \times \Sigma$ that is identically $1$ around $\Sigma_{0}$ and $\Sigma_{1}$ and $(g_t)_{t\in[0,1]}$ is a smooth family of Riemannian metrics on $\Sigma$. 
The action of $\Gamma$ on $M$ is induced by a $\Gamma$-action on $\Sigma$ via 
\begin{equation} \label{eq: group_action_Cauchy_surface}
    \gamma.(t,x')=(t,\gamma.x').
\end{equation}
Define a Riemannian metric on $M$ by 
\begin{equation} \label{eq: spacetime_Riem_metric}
    \hat{g} := N^2dt^2+g_t.
\end{equation}
Note that $\hat g$ is also invariant under the group action: as $g$ and $dt$ are invariant and orthogonality to slices is preserved, $N$ and $g_t$ must also be invariant and so $\hat g$ is invariant. 

We now briefly describe the Dirac bundle over a Lorentzian and Riemannian manifolds in a unified fashion, and refer to, 
e.g.~\citep{BSglobal} 
or~\citep{vDsplit} 
for details. 
The latter can be either $(M, \hat{g})$ or $(\Sigma, g_{t})$.
To deal with ${\ri}$-factors that differ between both cases, we let $\tau$ denote $1$ in the Riemannian and ${\ri}$ in the Lorentzian case. 

Let $\nu$ denote the past-directed unit normal field to $\Sigma_t$.  
The spinor bundle $SM$ of $M$ splits into bundles of right- and left-handed spinors, denoted by $S^+M$ and $S^-M$. We identify $S^+M|_{\Sigma_t}$ and $S\Sigma_t$ such that Clifford multiplications $\fc$ and $\fc_{\Sigma}$ on $M$ and $\Sigma_t$, respectively, are related by 
(see 
e.g.~\citep[Section 2.3]{vDsplit}) 
\[\mathsf{c}_\Sigma(X)=\tau\mathsf{c}(\nu)\mathsf{c}(X)\]
for any tangent vector $X$ on $M$. 
Note 
that~\citet{vDsplit} uses $\nu$ future-directed. 
This also leads to different choices for other identifications, so in the end everything ends up the same except for the sign of $\partial_t$ being opposite in every formula. 

In the Lorentzian case, the indefinite pairing $\<\cdot,\cdot\>_M$ of $SM$ is related to the inner product $\<\cdot,\cdot\>_{\Sigma}$ of $S\Sigma_t$ via
\[\<\cdot,\cdot\>_{\Sigma}=\<\cdot,\mathsf{c}(\nu)\cdot\>_M.\] 
Clifford multiplication on $SM$ is self-adjoint with respect to the pairing of $M$.
In the Riemannian case, both pairings coincide and Clifford multiplication is always skew-adjoint.
The Levi-Civita spin connections of $\spinorBun$ and $S \Sigma_{t}$ on $S^+M|_{\Sigma_t} \cong S\Sigma_t$ are related via 
(see 
e.g.~\citep[(2.9)]{vDsplit})
\begin{equation}
    \nabla^{SM}_X=\nabla^{S\Sigma_t}_X+\mathsf{c}(\nu)\mathsf{c}(\mathscr{W}X),\label{conneq}
\end{equation}
where $\mathscr{W}$ denotes the Weingarten map of $\Sigma_t$.

We additionally consider a complex twisting vector bundle $E$ equipped with a Hermitian  metric, a compatible connection $\nabla^{E}$ and an action of $\Gamma$ that covers the $\Gamma$-action on $M$ and preserves the connection and the hermitian metric. We define the twisted spinor bundle
\[F:=SM\otimes E\]
with grading
\[F^\pm:=S^\pm(M)\otimes E.\]
We obtain Clifford multiplications as above on $F$ by acting in the first factor and also obtain pairings on $F^+$, defined from the pairings above via 
\[\<s\otimes v, s'\otimes v'\>=\<s,s'\>\<v,v'\>.\] Whenever we talk about $L^2$-spaces of sections, these will be defined via the positive definite product $\<\cdot,\cdot\>_\Sigma$.

All the properties listed above also hold after twisting with $E$, and in the following, we will always use the above notation to refer to the operations on $F$.
The natural connection on $F$ is given by
\[\nabla^{F} (s\otimes v)=(\nabla^{SM}s)\otimes v+s\otimes(\nabla^Ev)\]
and it is compatible with the inner product of $F$. 
Formula \eqref{conneq} holds verbatim for the twisted connections (as the additional summand from the twisting is the same on either side of the equation).

The twisted Dirac operator (on positive chirality spinors) is given by 
\[D:= \sum\limits_{i=0}^{d} g (e_i,e_i) \mathsf{c}(e_i)\nabla_{e_i}^{F} : C^{\infty} (M; F^+) \to C^{\infty} (M; F^{-}),\]
where $(e_{i})$ is a local orthonormal frame of $TM$ and $C^{\infty} (M; F^\pm)$ is the space of smooth sections of the twisted spinor bundle $F^\pm$. 
In the following section, we want to treat the Lorentzian and Riemannian Dirac operators in parallel. 
Thus we allow $g$ to be either our Lorentzian metric~\eqref{eq: spacetime_metric} or the Riemannian metric~\eqref{eq: spacetime_Riem_metric}, and denote the twisted Dirac operators for either metric by $D$. 
Later on, we will distinguish them by writing $D$ and $\hat{D}$, respectively.

Let $A(t)$ denote the twisted Dirac operator on a hypersurface $\Sigma_t$. Note that we identified $(F^+)|_{\Sigma_t}$ with $S\Sigma_t\otimes (E|_{\Sigma_t})$, so the family $(A(t))_{t\in [0,1]}$ acts on sections of $F^+$ via
\[(Af)(t,x)=(A(t)f(t,\cdot))(x).\]
Note that $A (t)$ has discrete real spectrum as $\Sigma$ is compact.  
The equivariant $\eta$-invariant $\eta_{\gamma} \big( A (t) \big)$ of $A (t)$ is defined as the analytic continuation of the spectral function 
\begin{equation*}
    \eta_{\gamma} \big( s, A (t) \big) :=  \sum_{\lambda \in \spec A (t)} \sign (\lambda) |\lambda|^{-s} \tr_{E_{\lambda} (A (t))} (\gamma |_{E_{\lambda} (A (t))})  
\end{equation*}
at $s=0$ 
(\citep[p. 892]{Donnelly_IndianaUMJ_1978}). 

In the following, we shall be identifying functions on a product with iterated functions, i.e. writing $f(t)$ for $f(t,\cdot)$. We will also occasionally identify $\Sigma$ and $\Sigma_0$.

Note that everything constructed above is preserved by the $\Gamma$-action, as it is constructed from building blocks that are preserved.

\section{Relating to an abstract model operator}
\label{sec: relation_abstract_model_op}
Section \ref{sec: G_ind_G_sf} tells us how to reduce equivariant problems to the non-equivariant case: decompose everything into eigenspaces. While this works well on the functional analytic side, this does not work well on the geometric side: the restriction of a Dirac operator to a closed subspace is not again a Dirac operator. We thus want to relate our Dirac operator to operators of the form $\partial_t-{\ri}A$  (and do something similar for a suitable Riemannian Dirac operator). Operators of the latter form stay of that form under suitable restrictions, which will prove useful to us later.

If the lapse function $N$ were constantly one, we could get the result we want directly from \cite[Proposition 3.5]{vDsplit} (in that case, we would automatically have that the normal field $\nu$ is parallel). We can always obtain a metric with $N=1$ via continuous deformation of $g$. However, it is at this point not trivial to show that this deformation would preserve the equivariant index. It can be done independently, but we believe that the conversion between the geometric and the abstract setting (without a restriction on the lapse function) is interesting and potentially useful in its own right.

\subsection{Suitable isometries}
While $M$ is topologically a product, the inner products that occur on different timeslices do not coincide. To account for this, we have to carefully choose how we relate spinors at different times. Doing this in the right way will cancel out error terms later on.
\begin{df}
\label{dfisos}
    Let $\rho(t)$ denote the positive function on $\Sigma$ such that
    \[\dVol_{\Sigma_t}=\rho(t)^2\dVol_{\Sigma_0}.\]
    Let $\Pi_t$ denote parallel transport along timelines from $\Sigma_t$ to $\Sigma_0$ (by slight abuse of notation, we will use the same symbol for parallel transport on $TM$ and $F$). Let $\tilde\nu(t):=\Pi_t\nu(t)$. Let $\Phi(t)$ be the solution to the differential equation
    \[\partial_t\Phi(t)=-\frac{1}{2}\mathsf{c}(\tilde\nu(t))\mathsf{c}(\partial_t\tilde\nu(t))\Phi(t)\]
    with
    \[\Phi(0)=id.\]
    Define 
    \[U(t):=\Phi^{-1}(t)\Pi_t, \quad V(t):=\rho(t)U(t), \quad W(t):=N(t)^{\frac{1}{2}}V(t).\]
\end{df}

Note that all the objects defined above are constructed only from $g$-equivatiant ingredients and thus $g$-equivariant themselves. Similarly, everything preserves the grading (Clifford multiplication only occurs as a pair) and thus descends to $S^+M$.

\begin{lemma}
With respect to the pairings induced by either $\<\cdot,\cdot\>_\Sigma$ or $\<\cdot,\cdot\>_M$,
    \begin{enumerate}
        \item $U(t)$ is pointwise unitary and $U(t)\mathsf{c}(\nu(t))U(t)^{-1}=\mathsf{c}(\nu(0))$. 
        \item $V(t)$ as a map $L^2(\Sigma_t; S^+M)\to L^2(\Sigma_0; S^+M)$ is unitary.
        \item $W$ as a map $L^2(M; S^+M)\mapsto L^2([0,1]\times \Sigma_0; \pr^*(S^+M|_{\Sigma_0}))$, defined by $(W\phi)(t,x):=W(t)(\phi(t,x))$, is unitary. 
    \end{enumerate}
\end{lemma}
\begin{proof}
Parallel transport is an isometry with respect to the fibrewise pairing of $M$. We have 
\[\<\tilde\nu,\partial_t\tilde\nu\>_M=\frac{1}{2}\partial_t\<\tilde\nu,\tilde\nu\>=0,\]
as $\nu$ has unit length. Thus $\mathsf{c}(\tilde\nu)$ and $\mathsf{c}(\partial_t\tilde\nu(t))$ anticommute. This implies that $\mathsf{c}(\tilde\nu(t))\mathsf{c}(\partial_t\tilde\nu(t))$ is skew-adjoint with respect to the pairing of $M$ in either signature. We compute for pointwise spinors $v,w$
\begin{align*}
    \partial_t\<\Phi v,\Phi w\>_M&=-\frac{1}{2}(\<\mathsf{c}(\tilde\nu(t))\mathsf{c}(\partial_t\tilde\nu(t))\Phi v,\Phi w\>_M+\<\Phi v,\mathsf{c}(\tilde\nu(t))\mathsf{c}(\partial_t\tilde\nu(t))\Phi w\>_M)\\
    &=-\frac{1}{2}(\<\mathsf{c}(\tilde\nu(t))\mathsf{c}(\partial_t\tilde\nu(t))\Phi v,\Phi w\>_M-\<\mathsf{c}(\tilde\nu(t))\mathsf{c}(\partial_t\tilde\nu(t))\Phi v,\Phi w\>_M)\\
    &=0.
\end{align*}
As $\Phi(0)$ is the identity, we can conclude that all $\Phi(t)$ are isometries with respect to the pairing of $M$. Thus also all $U(t)$ are isometries for this pairing. In the Riemannian case, this is all we need, in the Lorentzian case we also have the pairing of $\Sigma$ to consider. For this we compute (using again that $\mathsf{c}(\tilde\nu)$ and $\mathsf{c}(\partial_t\tilde\nu(t))$ anticommute):
\begin{align*}
    &\partial_t(U(t)\mathsf{c}(\nu(t))U^{-1}(t))\\
    &=\partial_t(\Phi(t)^{-1}\Pi_t\mathsf{c}(\nu(t))\Pi_t^{-1}\Phi(t))\\
    &=\partial_t(\Phi(t)^{-1}\mathsf{c}(\tilde \nu(t))\Phi(t))\\
    &=-\Phi(t)^{-1}(\partial_t\Phi(t))\Phi(t)^{-1}\mathsf{c}(\tilde \nu(t))\Phi(t)+\Phi(t)^{-1}(\partial_t \mathsf{c}(\tilde \nu(t)))\Phi(t)+\Phi(t)^{-1}\mathsf{c}(\tilde \nu(t))(\partial_t\Phi(t))\\
    &=\frac{1}{2}\Phi(t)^{-1}\mathsf{c}(\tilde\nu(t))\mathsf{c}(\partial_t\tilde\nu(t))c(\tilde \nu(t))\Phi(t)+\Phi(t)^{-1} \mathsf{c}(\partial_t\tilde \nu(t))\Phi(t)\\
    &-\frac{1}{2}\Phi(t)^{-1}\mathsf{c}(\tilde \nu(t))\mathsf{c}(\tilde\nu(t))\mathsf{c}(\partial_t\tilde\nu(t))\Phi(t)\\
    &=\Phi(t)^{-1}\left(-\frac{1}{2}\mathsf{c}(\partial_t\tilde\nu(t)))+\mathsf{c}(\partial_t\tilde\nu(t)))-\frac{1}{2}\mathsf{c}(\partial_t\tilde\nu(t)))\right)\Phi(t)\\
    &=0
\end{align*}
Thus we have for all $t\in[0,1]$
\[U(t)\mathsf{c}(\nu(t))U^{-1}(t)=U(0)\mathsf{c}(\nu(0))U^{-1}(0)=\mathsf{c}(\nu(0)).\]
Using that $U$ is an isometry with respect to the pairing of $M$, we obtain for any two spinors $v,w$ over some point in $\Sigma_0$:
\begin{align*}
    \<U(t)^{-1}v,U(t)^{-1}w\>_{\Sigma}&=\<U(t)^{-1}v,\mathsf{c}(\nu(t))U(t)^{-1}w\>_{M}\\
    &=\<v,U(t)\mathsf{c}(\nu(t))U(t)^{-1}w\>_{M}\\
    &=\<v,\mathsf{c}(\nu(0))w\>_{M},\\
    &=\<v,w\>_{\Sigma}
\end{align*}
so $U$ is indeed an isometry for either pairing.

   Continuing with either pairing, we have for sections $\phi,\psi\in L^2(\Sigma_t; S^{+}M)$
    \begin{align*}
        \int\limits_\Sigma \<V(t)\phi(x),V(t)\psi(x)\>{\dVol}_{\Sigma_0}(x)&=\int\limits_\Sigma \<U(t)\phi(x),U(t)\psi(x)\>\rho(t,x)^2 {\dVol}_{\Sigma_0}(x)\\&=\int\limits_\Sigma \<\phi(x),\psi(x)\>{\dVol}_{\Sigma_t}.
    \end{align*}
    This proves the second claim. For the third claim, we observe that
    \[{\dVol}_M={\dVol}_{\Sigma_t}\wedge Ndt\]
    and thus for $\phi,\psi\in L^2(M; S^{+}M)$
    \begin{align*}
        &\int\limits_0^1\int\limits_\Sigma \<W(t)\phi(t,x),W(t)\psi(t,x)\>{\dVol}_{\Sigma_0}(x)dt\\ &=\int\limits_0^1\int\limits_\Sigma \<U(t)\phi(t,x),U(t)\psi(t,x)\>N(t,x)\rho(t,x)^2 {\dVol}_{\Sigma_0}(x)\\
        &=\int\limits_0^1\int\limits_\Sigma \<\phi(t,x),\psi(t,x)\>N(t,x){\dVol}_{\Sigma_t}(x)dt\\
        &=\int\limits_M \<\phi(z),\psi(z)\>{\dVol}_M(z).
    \end{align*}
\end{proof}
\begin{rem}
    In the Lorentzian case, $\Phi$ is needed to obtain isometries. In he Riemannian case, the use of $\Phi$ seems fairly arbitrary. It is nevertheless the right choice to make our later calculations work. 
\end{rem}
\subsection{Relating $D$ to the model operator}
\begin{lemma}
\label{gradN}
    We have
    \[\grad_\Sigma {N}=-\nabla_{\partial_t}\nu,\]
    where $\grad_\Sigma$ denotes the gradient as a function on $\Sigma_t$.
\end{lemma}
\begin{proof}
    We choose coordinates $(x^0,\dots,x^{d-1})$ on $\Sigma$ and obtain coordinates on $M$ by using $t$ as the $0$-coordinate.
    We have
    \[\nabla_{\partial_t}\nu=\nabla_{\partial_t}(N^{-1}\partial_t)=(\partial_tN^{-1})\partial_t +\sum\limits_{k=0}^{d-1}N^{-1}\Gamma^k_{00}\partial_k.\]
    As $\nu$ has constant length, we must have
    \[\<\nu,\nabla_{\partial_t} \nu\>=\frac{1}{2}\partial_t\<\nu,\nu\>=0,\]
    so only the parts orthogonal to $\nu$ remain:
    \[\nabla_{\partial_t} \nu=\sum\limits_{k=1}^{d-1}N^{-1}\Gamma^k_{00}\partial_k.\]
    We compute the relevant Christoffel symbols. For $k\neq 0$, we have
    \[\Gamma^k_{00}=\frac{1}{2}\sum\limits_{l=0}^{d-1}g^{kl}(-\partial_{x^l}g_{00})=-\frac{1}{2}\grad_\Sigma (N^2)_k=-N\grad_\Sigma (N)_k\]
    Thus
    \[\nabla_{\partial_t}\nu=-\grad_\Sigma {N}.\]
\end{proof}
\begin{cor}
$\partial_t\Phi(t)=0$ if $N$ is constant on $\Sigma_t$.
\end{cor}
\begin{proof}
    By the previous lemma, we obtain
    \[\partial_t\tilde \nu(t)=\pm N\nabla_\nu \nu(t)=0.\]
    By the chain rule, this implies the claim, as $\Phi(t)$ only depends on $t$ through $\tilde \nu(t)$.   
\end{proof}

The following is a generalisation of \cite[Proposition 3.5]{vDsplit} to the case where $\nabla_\nu\nu$ need not vanish.
\begin{thm}
    \label{Diractoabstract}
    We have
    \[WDW^{-1}=-\tau^2\mathsf{c}(\nu(0))(N^{-\frac{1}{2}}\partial_tN^{-\frac{1}{2}}+\tau^{-1}\tilde A)=-\tau^2\mathsf{c}(\nu(0))N^{-\frac{1}{2}}(\partial_t+\tau^{-1}N^\frac{1}{2}\tilde AN^\frac{1}{2})N^{-\frac{1}{2}},\]
    where $\tilde A(t):=V(t)A(t)V(t)^{-1}$ is unitarily equivalent to $A$ for each $t$(in a $\Gamma$-equivariant way).
\end{thm}
\begin{proof}
    The Dirac operator has the form (\cite[(2.10)]{vDsplit})
    \[D=\mathsf{c}(\nu)(\tau^2\nabla_\nu-\tau A-\frac{d-1}{2}H),\]
    where $H$ is the mean curvature of the $\Sigma_t$ and again $\tau$ is 1 in the Riemannian case and $\ri$ in the Lorentzian case (the proof in \cite{vDsplit} is for the non-twisted case but works verbatim with a twisting bundles, as \eqref{conneq} and all other identities used are the same as in the non-twisted case).
    We have
    \[W\mathsf{c}(\nu)W^{-1}=\Phi^{-1}\mathsf{c}(\tilde\nu)\Phi=\mathsf{c}(\nu(0)).\]
    and
    \begin{align*}
        W\nabla_\nu W^{-1}&=N^{\frac{1}{2}}\rho\Phi^{-1}\partial_\nu\circ\Phi\rho^{-1} N^{-\frac{1}{2}}\\
        &=N^{\frac{1}{2}}\partial_\nu\circ N^{-\frac{1}{2}}+N^{\frac{1}{2}}\rho\Phi^{-1}(\partial_\nu(\Phi\rho^{-1})) N^{-\frac{1}{2}}\\
        &=-N^{-\frac{1}{2}}(\partial_t+\rho\Phi^{-1}(\partial_t(\Phi\rho^{-1}))N^{-\frac{1}{2}}.
    \end{align*}
    Moreover
    \[WAW^{-1}=N^\frac{1}{2}\tilde AN^{-\frac{1}{2}}=N^{-\frac{1}{2}}(N\tilde A)N^{-\frac{1}{2}}=N^{-\frac{1}{2}}(N^{\frac{1}{2}}\tilde AN^{\frac{1}{2}}+N^{\frac{1}{2}}[N^{\frac{1}{2}},\tilde A])N^{-\frac{1}{2}}.\]
    Overall, we obtain
    \begin{align*}
        WDW^{-1}&= (W\mathsf{c}(\nu)W^{-1})^{-1}(\tau^2W\nabla_\nu W^{-1}-\tau WAW^{-1}-\frac{d-1}{2}H)\\
        &=-\tau^2 \mathsf{c}(\nu(0))^{-1}N^{{-\frac{1}{2}}}(\partial_t+\tau^{-1}(N^{\frac{1}{2}}\tilde AN^{\frac{1}{2}}+R))N^{{-\frac{1}{2}}},
    \end{align*}
    with
    \begin{align*}
        R(t)&:=N^{\frac{1}{2}}[N^{\frac{1}{2}},\tilde A]+\tau\Big(\rho\Phi^{-1}(\partial_t(\Phi\rho^{-1}))+\tau^2\frac{d-1}{2}NH\Big)\\
        &=\Big(N^{\frac{1}{2}}[N^{\frac{1}{2}},\tilde A]+\tau \Phi^{-1}\partial_t\Phi\Big) +\tau\Big(\rho(\partial_t\rho^{-1})+\tau^2\frac{d-1}{2}NH\Big)\\
    \end{align*}
    It remains to show that $R(t)=0$. For the first summand, we compute (using Lemma \ref{gradN}):
    \begin{align*}
        V^{-1}(N^{\frac{1}{2}}[N^{\frac{1}{2}},\tilde A]+\tau \Phi^{-1}\partial_t\Phi)V&=N^{\frac{1}{2}}[N^{\frac{1}{2}}, A]+\tau\Pi_t^{-1}(\partial_t\Phi)\Phi^{-1}\Pi_t\\
        &=-N^\frac{1}{2}\mathsf{c}_\Sigma(\grad_\Sigma(N^{\frac{1}{2}}))-\frac{1}{2}\tau\Pi_t^{-1}\mathsf{c}(\tilde{\nu})\mathsf{c}(\partial_t\tilde\nu)\Pi_t\\
        &=-\frac{1}{2}\mathsf{c}_\Sigma(\grad_\Sigma(N))-\frac{1}{2}\tau\mathsf{c}(\nu)\mathsf{c}(\nabla_{\partial_t}\nu)\\
        &=\frac{1}{2}\mathsf{c}_\Sigma(\nabla_{\partial_t} \nu)-\frac{1}{2}\mathsf{c}_\Sigma(\nabla_{\partial_t}\nu)\\
        &=0.       
    \end{align*}
    For the second summand, we replicate the computation of \cite[(3.5)]{vDsplit}. Let $(e_i)$ be an orthonormal frame with $e_0=\nu$ and choose coordinates on $\Sigma$ (which give coordinates on $M$). Let $\mathscr{W} (X):=-\tau^2\nabla_X(\nu)$ denote the Weingarten map. Using once again that $\<\nu,\partial_\nu \nu\>=0$, we compute
    \begin{align*}
        \tau^2(d-1)H&=\tau^2\tr(\mathscr{W} (X))\\
        &=-\sum\limits_{j=1}^{d-1} g(e_j,\nabla_{e_j}\nu)\\
        &=-\sum\limits_{j=0}^{d-1} g(e_j,\nabla_{e_j}\nu)\\
        &=-\operatorname{div}(\nu)\\
        &=-\det(g)^{-\frac{1}{2}}\partial_t(\det(g)^{\frac{1}{2}}(-N^{-1}))\\
        &=N^{-1}\det(g_t)^{-\frac{1}{2}}\partial_t(\det(g_t)^{\frac{1}{2}})\\
        &=N^{-1}\rho^{-2}\partial_t(\rho^2)\\
        &=2N^{-1}\rho^{-1}\partial_t\rho\\
        &=-2N^{-1}\rho\partial_t\rho^{-1}
    \end{align*}
    thus we have
    \[\rho\partial_t\rho^{-1}+\tau^2\frac{{d-1}}{2}NH=0\]
    and hence 
    \[R=0.\]
\end{proof}

%
%
%
%
%
%
%
%
%
%
\section{Equivariant index theorem} 
\label{sec: g_ind_thm}
\subsection{From Lorentzian to Riemannian Dirac index}

In order to reduce the Lorentzian equivariant index theorem to the Riemannian one, we want to show that the equivariant index of our Lorentzian Dirac operator is equal to that of the constructed Riemannian Dirac operator. To do this, we show that both indices are equal to the equivariant spectral flow of the slice Dirac operators. To disambiguate the Dirac operators, we will from now on write $\hat D$ for the Dirac operator associated to $\hat g$ and keep $D$ for the Lorentzian one. Note that the family $A$ of hypersurface Dirac operaotors is the same in the Riemannian and Lorentzian case. To get a well-defined index, we impose APS boundary conditions on our operators:

\begin{df}
    Let $S$ be an unbounded operator either in $L^2([0,T]; H)$ for some Hilbert space $H$ or in $L^2(M; S^+M)$ (in which case we set $T=1$). Let $(B(t))_{t\in[0,T]}$ be self-adjoint operators either on $H$ or on $L^2(\Sigma_t; S\Sigma_t)$. The closure of $S$ with APS boundary conditions for $B$ is defined as follows:
    take the closure of the operator $S$ (with the maximal domain it is reasonably defined on) and then restrict its domain to those functions/sections $f$ that satisfy $f(0)\in E_{(-\infty,0)}(B(0))$ and $f(T)\in E_{(0,\infty)}(B(T))$.
    
    Let $D(B,T)$ denote the closure of the operator $\partial_t-{\ri}B(t)$ in $L^2([0,T]; H)$ with APS boundary conditions for $B$. Let $D(B):=D(B,1)$. Let $\hat D(B)$ denote the closure of $\partial_t+B$ in $L^2([0,1]; H)$ with APS boundary conditions for $B$. 

    Let $D$ denote the twisted spin-Dirac operator associated to $g$ and $\hat D$ that associated to $\hat g$. We denote by $D_{\APS}$ resp. $\hat D_{\APS}$ the closure of $D$ resp. $\hat D$ with APS boundary conditions for $A$.
\end{df}
\begin{rem}
    A reader familiar with APS boundary conditions in the Riemannian case might wonder why we impose different conditions on the future and past boundary. In fact the APS conditions defined here are the usual APS conditions, the apparent difference comes from the fact that we identify $S^+M|_{\partial M}$ and $S\partial M$ in different ways. This identification depends on a choice of normal vector field. While the usual formulation of the Riemannian APS uses the outward pointing normal vector field, we use the past directed normal vector field, which gives us a different sign on the future boundary.
\end{rem}
We shall use the following theorem from the non-equivariant setting (see \cite[theorems 4.9 and 5.9]{vDR}).
\begin{thm}
    \label{indsf}
    Let $(B(t))_{t\in[0,1]}$ be a strongly continuously differentiable family of self-adjoint operators with a common domain. Then $\hat D(B)$ is Fredholm and
    \[\ind(\hat D(B))=\sfl(B)-\dim\ker(B(1)).\]
    If additionally $D(B,t)$ is Fredholm for all $t\in [0,1]$, then
    \[\ind (D(B))=\sfl(B)-\dim\ker(B(1)).\]
\end{thm}
\begin{rem}
    The extra term $-\dim\ker(B(1))$ compared to \cite{vDR} comes in because APS boundary conditions were defined with $E_{[0,\infty)}(A(1))$ instead of $E_{(0,\infty)}(A(1))$ for convenience in that reference.
\end{rem}
Using our preparations on equivariant index and spectral flow, we can obtain an equivariant version of this.
\begin{thm}
\label{indsfeq}
    Let $(B(t))_{t\in[0,1]}$ be a strongly continuously differentiable family of self-adjoint operators with a common domain in a Hilbert space $H$. Assume that all  $B(t)$ are equivariant with respect to the action of some group element $\gamma$. Then
    \[\ind_\gamma(\hat D(B))=\sfl_\gamma(B)-\tr(\gamma|_{\ker B(1)}).\]
    If additionally $D(B,t)$ is Fredholm for all $t\in [0,1]$, then
    \[\ind_\gamma (D(B))=\sfl_\gamma(B)-\tr(\gamma|_{\ker B(1)}).\]
\end{thm}
\begin{proof}
    Let $E_I(\gamma)$ denote the spectral subspaces of the action of $\gamma$ on $H$ and let $E_I'(\gamma)$ those of the action on $L^2([0,1]; H)$. The restriction of the closure of an operator to a closed invariant subspace is the closure of the restricted operator. Taking positive/negative subspaces of the boundary operator also commutes with restricting. Thus we have
    \[D(B|_{E_I(\gamma)})=D(B)|_{E_I'(\gamma)}\]
    and
    \[\hat D(B|_{E_I(\gamma)})=\hat D(B)|_{E_I'(\gamma)}.\]
    As restriction to a closed invariant subspace preserves Fredholmness, we know that
    \[D(B|_{E_I(\gamma)},t)\]
    is Fredholm for all $t\in [0,1]$. We can now combine theorems \ref{inddeco}, \ref{sfdeco}, and \ref{indsf} to calculate
    \begin{align*}
    \ind_\gamma(D(B))&=\sum\limits_{\lambda\in \Eig(\gamma)}\lambda\ind(D(B)|_{E_\lambda'(\gamma)})\\
    &=\sum\limits_{\lambda\in \Eig(\gamma)}\lambda\ind(D(B|_{E_\lambda(\gamma)}))\\
    &=\sum\limits_{\lambda\in \Eig(\gamma)}\lambda\sfl(B|_{E_\lambda(\gamma)})-\dim\ker(B(1)|_{E_\lambda(\gamma)})\\
    &=\sfl_\gamma(B)-\tr(\gamma|_{\ker B(1)}).
    \end{align*}
    The same calculation with hats shows 
    \[\ind_\gamma(\hat D(B))=\sfl_\gamma(B)-\tr(\gamma|_{\ker B(1)}).\]
\end{proof}
Now the work we did in relating the Dirac operators to abstract ones allows us to apply this theorem to Dirac operators as well.
\begin{thm}
    \label{indsfdirac}
    For any $\gamma\in \Gamma$, we have
    \[\ind_\gamma(D_{\APS})=\sfl_{\gamma} (A)-\tr(\gamma|_{\ker A(1)})=\ind_\gamma(\hat D_{\APS}).\]
\end{thm}
\begin{proof}
    By Theorem \ref {Diractoabstract}, we have
    \[WDW^{-1}=\mathsf{c}(\nu(0))N^{-\frac{1}{2}}(\partial_t-{\ri}B)N^{-\frac{1}{2}}\]
    for
    \[B=N^\frac{1}{2}\tilde AN^\frac{1}{2}\]
    and $\tilde A$ unitarily equivalent to $A$. As $A$ is a family of self-adjoint elliptic operators with smooth coefficients, so is $B$. In particular, $B$ is strongly continuously differentiable on the common domain $H^1(\Sigma; S\Sigma)$. As $N$ was constructed to be constantly $1$ near the boundary, multiplication by powers of $N$ preserves boundary conditions and $W$ maps sections with APS-boundary conditions for $A$ to sections with APS boundary conditions for $B$ bijectively. Thus we have
    \[WD_{APS}W^{-1}=\mathsf{c}(\nu(0))N^{-\frac{1}{2}}D(B)N^{-\frac{1}{2}}.\]
    As compositions with equivariant isomorphisms preserve the equivariant index we get (using Theorem \ref{indsfeq})
    \[\ind_\gamma(D_{\APS})=\ind_\gamma(D(B))= \sfl_{\gamma} (B) + \tr (\gamma|_{\ker B (1)}).\]
    Using again that $N$ is $1$ near the boundary, we get a continuous homotopy through self-adjoint elliptic $\gamma$-equivariant operators from $\tilde A$ to $B$ with fixed endpoints via
    \[s\mapsto N^{\frac{s}{2}}\tilde A N^{\frac{s}{2}}.\]
    As $\gamma$-spectral flow is preserved under fixed endpoint homotopies and equivariant unitary equivalence, we have
    \[\ind_\gamma(D_{\APS})=\sfl_\gamma(B)-\tr(\gamma|_{\ker B(1)})=\sfl_\gamma(\tilde A)-\tr(\gamma|_{\ker \tilde A(1)})=\sfl_\gamma(A)-\tr(\gamma|_{\ker A(1)}).\]
    The same argument with hats and $+B$ instead of $-{\ri}B$ yields
    \[\ind_\gamma(\hat D_{\APS})=\sfl_\gamma(A)-\tr(\gamma|_{\ker A(1)}).\]
\end{proof}
\begin{rem}
    The index $=$ spectral flow theorem in the Riemannian case has already been shown in \cite[Theorem 3.17]{HoYa}, where it is derived more generally in a K-theoretic framework. 
\end{rem}
As a corollary, we obtain the following:
\begin{cor} \label{cor: g_ind_Dirac_op_metric_deformation}
    The equivariant index of the Lorentzian Dirac operator $D_{\APS}$ is preserved under continuous deformations of the slice metrics $g_t$, the lapse function $\eta$ and the twisting connection $\nabla^E$, as long as all quantities are held constant on the boundary.
\end{cor}
\begin{proof}
    Such a deformation of the metric leads to a fixed endpoint homotopy for the associated hypersurface Dirac operators $A$. As this preserves the $\gamma$-spectral flow,
    \[\ind_\gamma(D_{\APS}) = \sfl_\gamma(A)-\tr(\gamma|_{\ker A(1)})\]
    remains constant under the deformation.
\end{proof}
\begin{rem}
    Unlike in the Riemannian case, the above does not follow immediately from deformation invariance of the index, as the domain of (the closure of) the Lorentzian Dirac operator varies with a change in the metric, independently from boundary conditions.
\end{rem}

%
%
%
%
%
%
%
%
%
%
\subsection{A Riemannian equivariant index theorem}
In this section, we introduce the characteristic forms that arise in both the Riemannian and the Lorentzian index theorems and present the version of the Riemannian index theorem that we will be using.

For this section and the next, we allow $M$ to be any compact Riemannian spin manifold with boundary that has a structure-preserving action of a group $\Gamma$ (in particular, this could be our $(M,\hat g)$ from before). Those definitions that also make sense on a Lorentzian manifold we will also use for our Lorentzian spin spacetime $(M,g)$.

Let $M_{\gamma} := \{ x \in M \,|\, x = \gamma . x \}$ be the fixed-point set of $\gamma \in \Gamma$. 
By, 
e.g.~\cite[Theorem 5.1]{Kobayashi_Springer_1995},~\cite[Theorem 1.10.15]{Klingenberg_deGruyter_1995}), 
$M_\gamma$ is a smooth, closed, totally geodesic submanifold of $M$ that is not necessarily connected and may have non-constant dimension. We denote the induced metric on $M_\gamma$ by $\fh$. One has the orthogonal splitting 
\begin{equation*}
    \rT \! M |_{M_{\gamma}} = \rT \! M_{\gamma} \oplus M_{\gamma}^{\perp}, 
\end{equation*}
where $M_{\gamma}^{\perp}$ is the normal bundle of $M_{\gamma}$. 
Since $M_{\gamma}$ is totally geodesic, the Levi-Civita connection on $M$ induces the Levi-Civita connection $\nabla$ on $M_{\gamma}$. 
Let $R$ denote the Riemann curvature of $M_\gamma $ and $R^{\perp}$ the curvature of $M_{\gamma}^{\perp}$ with respect to the restriction of the Levi-Civita connection of $M$.  
The $\hat{A}$-form of $M_{\gamma}$ is now defined by (see 
e.g.~\cite[p. 187]{Berline_Springer_2004}) 
\begin{equation} \label{eq: def_A_hat_genus_fixed_pt_set}
    \hat{A} (M_{\gamma}) := \hat{A} (\nabla) := \mathrm{det}^{1/2} \frac{R/2}{ \sinh (R/2)}. 
\end{equation}
Here $\det^{1/2}$ of a differential form that has to be understood as the Pfaffian of the form, see, 
e.g.~\cite[Proposition 1.36 and (1.29)]{Berline_Springer_2004} 
for details. 
%
%
%

For any vector bundle $\sE$, we get a map 
(see 
e.g.~\cite[Proposition 1.31 and Definition 1.32]{Berline_Springer_2004}) 
\begin{equation*}
    \Str_\mathscr{E} : C^{\infty} \big( M; \wedge \coTan M \otimes \End \mathscr{E} \big) \to C^{\infty} (M; \wedge \coTan M), 
    \quad 
    \Str_\mathscr{E} (\alpha \otimes a) := \alpha \Str a,  
\end{equation*}
where $\wedge \coTan M \to M$ is the bundle of exterior differential forms. 
The localised Chern character form of the twisting bundle $E$ on $M_\gamma$ is defined 
by (\cite[p. 190]{Berline_Springer_2004}) 
\begin{equation} \label{eq: def_Chern_char_form_twist_bundle}
    \ch_{\gamma} (\nabla^{E}) := \Str_{E} \big( \gamma^{E} \exp (- \Omega) \big),   
\end{equation}
where $\gamma^{E} \in \End (E |_{M_{\gamma}})$ and $\Omega$ is the curvature of $\nabla^{E}$ restricted to $M_{\gamma}$.

%
%
%
\begin{rem}
    The normalisation 
    (following~\cite[p. 47]{Berline_Springer_2004} 
    and~\cite[(11.17) and (11.22)]{Lawson_PUP_1989}) 
    of the $\hat{A}$-form and Chern character form used in this article differs from that preferred by topologists by a $- 2 \pi {\ri}$ factor. 
    In other words, the substitutions $R \mapsto {\ri} R / 2 \pi$ in~\eqref{eq: def_A_hat_genus_fixed_pt_set} and $\Omega \mapsto {\ri} \Omega / 2\pi$ in~\eqref{eq: def_Chern_char_form_twist_bundle} give the conventional definitions. 
\end{rem}
%
%
%

The form we are interested in, which will appear as the integrand in the index theorem is the following.
\begin{df}
    We  define
    \[{\ka}  
    := \frac{\hat{A} (\nabla) \wedge \ch_{\gamma}(\nabla^{E})}{\det^{1/2} \big( \one - \gamma^{\perp} \exp (- {R}^{\perp}) \big)}.\]
\end{df}

%
%
%
\begin{rem} \label{lem: BGVD_ind_thm_spin}
    If $M_{\gamma}$ has a spin-structure then 
    (cf.~\citealp[p. 193]{Berline_Springer_2004}) 
    \begin{equation*}
        \ka = \ri^{-\ell^{\perp}} \frac{\hat{A} (M_{\gamma}) \wedge \ch_{\gamma} (\nabla^E)}{\ch_{\gamma} (\sS M_{\gamma}^{\perp})},  
    \end{equation*}
    where $\sS M_{\gamma}^{\perp}$ is the spinor bundle over $M_{\gamma}^{\perp}$. 
\end{rem}
%
%
%

To compute the index of $\tilde D$, we need a suitable version of the Riemannian index theorem as given below.
\begin{thm}[\citet*{Berline_Springer_2004},~\cite{Donnelly_IndianaUMJ_1978}]
\label{thm: BGVD_ind_thm}
Let $M$ be a smooth compact Riemannian spin manifold with boundary $\partial M$. Let $E$ be a smooth hermitian vector bundle with a compatible connection. Assume that there is a bundle map $\gamma$ that descends to an isometry on $M$ and also preserves all further structures. Assume furthermore that all structures are of product form near the boundary.  Let $\slashed D$ be the twisted Dirac operator on positive chirality spinors. Let $A_{\partial M}$ be the Dirac operator on the boundary, with the spin structure induced via the outward pointing normal vector field. Then we have
\begin{equation*}
    \gInd (\slashed{D}_{\APS}) 
    = 
    \int_{M_{\gamma}} (2 \pi \ri)^{- \ell} \ri^{-\ell^{\perp}} \ka 
    - 
    \frac{1}{2}(\eta_\gamma(A_{\partial M})+\tr(\gamma|_{\ker(A_{\partial M})})), 
\end{equation*}
where $\ell:=\frac{\dim(M_\gamma)}{2}$ and $\ell^\bot:=\frac{\mathrm{codim} (M_\gamma)}{2}$.
\end{thm}
%
%
%
\begin{rem} \label{rem: orientation_fixed_pt_set}
The action of $\gamma$ induces an orientation on $M_\gamma$
(see ~\cite[Proposition 6.14]{Berline_Springer_2004}). 
In Theorem~\ref{thm: global_g_ind} and Theorem~\ref{thm: BGVD_ind_thm}, we have endowed $M_{\gamma}$ with that orientation. 
If $M_\gamma$ is equipped with some other natural orientation (e.g. if it is zero-dimensional) the change in orientation may introduce further signs in the integrand (c.f. ~\cite[p. 191 and 193]{Berline_Springer_2004}).
\end{rem}
%
%
%
\begin{proof}
\citet{Donnelly_IndianaUMJ_1978} uses a standard doubling construction to embed $M$ into a closed manifold and considers the heat kernels $f_1$ resp. $f_2$ of the operators $\slashed{D}^*\slashed{D}$ resp. $\slashed{D} \slashed{D}^*$ on the double. He then shows (\citep[Theorem 1.2 and (1.1)]{Donnelly_IndianaUMJ_1978}) that 
\begin{equation}
    \gInd (\slashed{D}_{\APS}) = \mathfrak{A} - 
    \frac{1}{2}(\eta_\gamma(A_{\partial M})+\tr(\gamma|_{\ker(A_{\partial M})})), 
\end{equation}
where $\mathfrak{A}$ is the constant part in the small-$t$ asymptotic expansion of 
\[\int\limits_M \tr(\gamma^{-1}f_1(x,\gamma x,t))-\tr(\gamma^{-1}f_2(x,\gamma x,t)))dx.\]

Now we employ~\cite[Theorem 6.11]{Berline_Springer_2004}. We use the notation $\Phi_j$ as in that theorem and $\sigma$ as in ~\cite[Definition 3.31]{Berline_Springer_2004} and denote by $\dVol^\bot$ the unique form such that $\dVol_M=\dVol_{M_\gamma}\wedge \dVol^\bot$. Applying ~\cite[proposition 3.21 and 6.15, Definition 6.13]{Berline_Springer_2004} \footnote {Note that there is a factor of $\dVol^\bot$ missing on the right hand side of \citep[Proposition 6.15]{Berline_Springer_2004}. Also note that we chose the orientation of $M_\gamma$ such that the factor $\varepsilon(\gamma)$ in \citep{Berline_Springer_2004} is 1.}, as well as the fact that integration only sees the top degree part of a form, allows us to calculate (as in \citep{Berline_Springer_2004})
\begin{align*}
    \mathfrak{A}&=\int\limits_M(4\pi)^{-\ell} \mathrm{Str} (\Phi_\ell(\gamma))\dVol\\
    &=\int\limits_M (4\pi)^{-l}(-2 \ri)^\frac{d}{2}\sigma_d(\Phi_\ell(\gamma))\\
    &=\int\limits_M (4\pi)^{-l}(-2 \ri)^\frac{d}{2}2^{-\ell^\bot}\ka\wedge\dVol^\bot\delta_{M_\gamma}\\
    &=\int\limits_{M_\gamma}(2\pi \ri)^{-\ell}i^{-\ell^\bot} \ka. 
\end{align*}
\end{proof}
%
%
%
%
%
%
%
%
%
%
\subsection{Transgression forms}

The index theorem above
requires all structures to be of product form near the boundary. Thus we introduce comparison structures that satisfy this requirement. Following~\citet{BSglobal}, 
we consider a deformation of the metric~\eqref{eq: spacetime_Riem_metric} near the boundary $\partial M$ to a metric $\tilde{g}$ with product structure near the boundary, i.e.,
\begin{equation*}
    \tilde{g} = \rd t^{2} + g_{j} 
\end{equation*}
near $\Sigma_{j}$ where $j = 0, 1$. Let ${\langle \cdot, \cdot \rangle}'_E$ and $\tilde\nabla^E$ be a hermitean product and a compatible connection on $E$, that are of product structure near $\partial M$ and coincide with the original structures on $\partial M$. These can be constructed as follows: take any metric and connection in the interior, take the product metric and product connection near the boundary and glue these with a partition of unity. This will yield a hermitian metric ${\langle \cdot, \cdot \rangle}_E'$ and a connection $\nabla'$ that is product and compatible near the boundary, but not compatible everywhere. We fix this using a trick from \citet[(2.13)]{HuLe}. We let $\nabla'^{*}$ denote the adjoint connection of $\nabla'$ defined via
\[{\langle \nabla'^{*}_Xf, g \rangle}'_E=\partial_X{\langle f, g \rangle}'_E-{\langle f, \nabla'_X g \rangle}'_E.\]
Then the connection
\begin{equation}
\label{adjointconn}
\tilde\nabla^E:=\frac{1}{2}(\nabla'+\nabla'^{*})
\end{equation}
is metric compatible and coincides with $\nabla^0$ near the boundary, i.e. has product structure there.
We denote the twisted Dirac operator corresponding to these new structures by $\tilde{D}$. Near the $\Sigma_{j}$ it is of the form 
\begin{equation*} \label{eq: Dirac_op_bdy}
    \tilde{D} = - \fc (\nu) (\partial_{t} + A_{j}). 
\end{equation*}

Let $\tilde{\fh}$ be the Riemannian metric on $M_{\gamma}$ induced from $\tilde{g}$ and $\tilde{\nabla}$ the Levi-Civita connection on $(M_{\gamma}, \tilde{\fh})$ induced from the Levi-Civita connection on $(M, \tilde{g})$. 
We denote the corresponding $\hat{A}$-form by $\hat{A} (\tilde{\nabla})$. 
Analogously, $\ch_{\gamma} (\tilde{\nabla}^{E})$ denotes the localised Chern character form of the product connection $\tilde{\nabla}^{E}$ on the twisting bundle $E$.

The integrand $\ka$ is comprised of three different forms: the $\hat A$-form, the localized Chern character and the term coming from the normal bundle of the fixed point set.
To compare these quantities between the original set-up and the new one with product structure, we shall use what is known as transgression forms. By a transgression form for some curvature dependent form $\alpha$, we will mean a form $\fT \alpha$ associated to two connections such that 
\[\alpha(\nabla_1)-\alpha(\nabla_2)=\rd \, \fT \alpha(\nabla_2,\nabla_1),\]
where $\rd$ denotes the exterior derivative. For any such transgression form, we have by Stokes theorem
\begin{equation} \label{eq: int_transgression_form_ASS_integrand}
    \int\limits_{\partial M}\fT\alpha(\nabla_1,\nabla_2)=\int\limits_M \rd \, \fT\alpha(\nabla_1,\nabla_2)=\int\limits_M\alpha(\nabla_2)-\int\limits_M\alpha(\nabla_1). 
\end{equation}
In particular, the boundary integral does not depend on the choice of transgression form. 

Our goal is to compute the transgression forms for the forms in the integrand of the equivariant index theorem. This will work analogously to the computation for standard characteristic forms 
(see~\citet[Proposition 1.41 and Theorem 7.7]{Berline_Springer_2004}). 
Let $\sE$ be  a vector bundle with a connection $\nabla^{\sE}$ and $C^{\infty}(M; \wedge^{\even} \coTan M \otimes \End \sE)$ the space of even degree $\End \sE$-valued differential forms on $M$. 
In what follows, we need a slight generalisation of an identity used in the proof of the transgression formula 
in~\citet[Proposition 1.41]{Berline_Springer_2004}. 
\begin{lemma} \label{lem: str_alpha_beta_gamma}
Let $\alpha_\cdot \colon [0,1]\to \wedge^{\even} \coTan M \otimes \End \sE.$ If $\beta,\gamma\in \wedge^{\even} \coTan M \otimes \End \sE$ commute with $\alpha_s$, then for any polynomial $f$, we have
    \[\Str_\sE(\beta\partial_s f(\alpha_s)\gamma)=\Str_\sE(\beta(\partial_s\alpha_s)f'(\alpha_s)\gamma).\] 
\end{lemma}
\begin{proof}
    As even forms commute and the trace is cyclic on $\End \sE$, we have for $a,b\in \wedge^{\even} \coTan M \otimes \End \sE$:
    \[\Str_\sE(ab)=\Str_\sE(ba).\]
    By linearity, it suffices to show the Lemma for monomials $f(z)=z^k$. For these, we have
    \begin{align*}
        \Str_\sE(\beta\partial_s f(\alpha_s)\gamma)&=\Str_\sE(\beta\sum\limits_{j=1}^k\alpha_s^{k-j}\partial_s \alpha_s\alpha_s^{j-1}\gamma)\\
        &=\sum\limits_{j=1}^k\Str_\sE(\beta(\partial_s \alpha_s)\alpha_s^{k-1}\gamma)\\
        &=\Str_\sE(\beta(\partial_s \alpha_s)f'(\alpha_s)\gamma). 
    \end{align*}
\end{proof}
%
%
%

To determine $\fT \! \hat{A} (\tilde{\nabla}, \nabla)$, we set $\nabla_{s} := (1 - s) \nabla + s \tilde{\nabla}$ for $s \in [0, 1]$. 
Let $R_{s}$ be the respective curvature with respect to $\nabla_{s}$. 
Note that~\eqref{eq: def_A_hat_genus_fixed_pt_set} can be expressed as 
\begin{equation*}
    \hat{A} (\nabla_{s}) = \exp \Big( \Str \big( f_{\hat{A}} (R_{s}) \big) \Big), 
    \quad 
    f_{\hat{A}} (R_{s}) := \frac{1}{2} \ln \frac{R_{s}/2}{\sinh (R_{s}/2)}. 
\end{equation*}
Now, performing the $\rd / \rd s$ derivative in $\hat{A} (\tilde{\nabla}) - \hat{A} (\nabla) = \int_{0}^{1} \rd \hat{A} (\nabla_{s})/ \rd s$ using Lemma~\ref{lem: str_alpha_beta_gamma} results in the transgression form of the $\hat{A}$-form 
\begin{equation} \label{eq: transgression_form_A_hat}
    \fT \! \hat{A} (\tilde{\nabla}, \nabla) := \frac{\rd}{2} \int_{0}^{1} \hat{A} (\nabla_{s}) \Str_{\tangent M_\gamma} \left( \frac{\rd \nabla_{s}}{\rd s} \frac{\rd}{\rd s} \ln \frac{R_{s} / 2}{ \sinh (R_{s} / 2)} \right) \rd s.  
\end{equation}

We now proceed as before by setting $\nabla_{s}^{E} := (1 - s) \nabla^{E} + s \tilde{\nabla}^{E}$ for $s \in [0, 1]$ to determine $\fT \ch_{\gamma} (\tilde{\nabla}^{E}, \nabla^{E})$. 
Let $\Omega_{s}$ be the respective curvature with respect to $\nabla_{s}^{E}$. 
Then employing Lemma~\ref{lem: str_alpha_beta_gamma}, we obtain  
\begin{align*}
    \ch_{\gamma} (\tilde{\nabla}^{E}) - \ch_{\gamma} (\nabla^{E}) 
    & = 
    \int_{0}^{1} \frac{\rd \ch_{\gamma} (\tilde{\nabla}_{s})}{\rd s} \rd s
    \nonumber \\ 
    & = 
    \int_{0}^{1} \Str_{E} \left( \gamma^{E} \frac{\rd \exp (- \Omega_{s})}{\rd s} \right)
    \nonumber \\  
    & = 
    - \int_{0}^{1} \Str_{E} \left( \gamma^{E} \frac{\rd \Omega_{s}}{\rd s} \exp (- \Omega_{s})  \right)
    \nonumber \\  
    & = 
    - \int_{0}^{1} \Str_{E} \left( \gamma^{E} \left[ \nabla_{s}^{E}, \frac{\rd \nabla_{s}^{E}}{\rd s} \exp (- \Omega_{s}) \big) \right] \right)
    \nonumber \\
    & = 
    - \int_{0}^{1} \Str_{E} \left(  \left[ \nabla_{s}^{E}, \gamma^{E}\frac{\rd \nabla_{s}^{E}}{\rd s} \exp (- \Omega_{s}) \big) \right] \right)
    \nonumber \\ 
    & = 
    -
    \int_{0}^{1} \rd \Str_{E} \left( \gamma^{E} \frac{\rd \nabla_{s}^{E}}{\rd s} \exp (- \Omega_{s})  \right), 
\end{align*}
where we have used the fact that $\Omega_{s} = (\nabla_{s}^{E})^2$ 
and ~\cite[Lemma 1.42]{Berline_Springer_2004}).
Therefore, the transgression form of the Chern character form is given by 
\begin{equation} \label{eq: transgression_form_Chern}
    \fT \! \ch_{\gamma} (\tilde{\nabla}^{E}, \nabla^{E}) := - \int_{0}^{1} \Str_{E} \left( \gamma^{E} \omega \exp (- \Omega_{s}) \right) \rd s.  
\end{equation}

Next, we compute the transgression form of 
\[{\det}^{-1/2} \big( \one - \gamma^{\perp} \exp (- R^{\perp}) \big) = \exp (- \frac{1}{2} \Str_{M_\gamma^\bot} \ln f^{\perp}),\]
where $f^{\perp} := \one - \gamma^{\perp} \exp (- R^{\perp})$.
As before, set $\nabla_{s}^{\perp} := (1 - s) \nabla^{\perp} + s \tilde{\nabla}^{\perp}$ for $s \in [0, 1]$. 
Employing Lemma~\ref{lem: str_alpha_beta_gamma} twice, we get 
\begin{align*}
    \exp (- \frac{1}{2} \Str_{M_\gamma^\bot} \ln \tilde{f}^{\perp}) 
    & - \exp (- \frac{1}{2} \Str_{M_\gamma^\bot} \ln f^{\perp}) 
    \nonumber \\ 
    & = 
    - \frac{1}{2} \int_{0}^{1} \frac{\rd \Str_{M_\gamma^\bot} \ln f_{s}^{\perp}}{\rd s} \exp (- \frac{1}{2} \Str_{M_\gamma^\bot} \ln f_{s}^{\perp}) \rd s 
    \nonumber \\ 
    & = 
    - \frac{1}{2} \int_{0}^{1} \Str_{M_\gamma^\bot} (-\gamma^\bot\frac{\rd \exp(-R^\perp_s)}{\rd s} (f_{s}^{\perp})^{-1]}) \exp (- \frac{1}{2} \Str_{M_\gamma^\bot} \ln f_{s}^{\perp}) \rd s 
    \nonumber \\ 
    & = 
    \frac{1}{2} \int_{0}^{1}  \Str_{M_\gamma^\bot} \left( \frac{\rd R_{s}^{\perp}}{\rd s} \frac{\gamma^{\perp} \exp (- R^{\perp})}{f_{s}^{\perp}} \right) \exp (- \frac{1}{2} \Str_{M_\gamma^\bot} \ln f_{s}^{\perp}) \rd s 
    \nonumber \\ 
    & = 
    \frac{1}{2} \int_{0}^{1} \Str_{M_\gamma^\bot} \left( \left[ \nabla_{s}^{\perp}, \frac{\rd \nabla_{s}^{\perp}}{\rd s} \frac{\one - f_{s}^{\perp}}{f_{s}^{\perp}} \exp (- \frac{1}{2} \Str_{M_\gamma^\bot} \ln f_{s}^{\perp}) \right] \right) 
     \rd s 
    \nonumber \\ 
    & = 
    \frac{1}{2} \rd \int_{0}^{1} \Str_{M_\gamma^\bot} \left( \frac{\rd \nabla_{s}^{\perp}}{\rd s} \frac{\one - f_{s}^{\perp}}{f_{s}^{\perp}} \exp (- \frac{1}{2} \Str_{M_\gamma^\bot} \ln f_{s}^{\perp}) \right) 
     \rd s. 
\end{align*}
This gives 
\begin{equation} \label{eq: transgression_form_Riem_curv_perp}
    \fT \exp (- \frac{1}{2} \Str_{M_\gamma^\bot} \ln f^{\perp}) 
    = 
    \frac{1}{2} \int_{0}^{1} \Str_{M_\gamma^\bot} \left( \frac{\rd \nabla_{s}^{\perp}}{\rd s} \frac{\one - f_{s}^{\perp}}{f_{s}^{\perp}} \exp (- \frac{1}{2} \Str_{M_\gamma^\bot} \ln f_{s}^{\perp}) \right) 
     \rd s.
\end{equation}

A transgression form of a product of forms is given by 
(see~\cite[(6.12)]{Braverman_IndianaUMJ_2019})
\[\fT(\alpha\wedge\beta)(\nabla_1,\nabla_1',\nabla_2,\nabla_2'):=\alpha(\nabla_2)\wedge\fT\beta(\nabla_1',\nabla_2')+\fT\alpha(\nabla_1,\nabla_2)\wedge\beta(\nabla_1').\]
We combine the above results to obtain the following.

%
%
%
\begin{lemma} \label{lem: transgression_form_ASS}
    The transgression form $\fT \ka$ of the integrand $\ka$ in Theorem~\ref{thm: global_g_ind} is given by 
    \begin{align}
        \fT \ka 
        & = 
        \fT \big( \exp (- \frac{1}{2} \Str_{M_\gamma^\bot} \ln f^{\perp}) \big) \wedge \hat{A} (\tilde{\nabla}) \wedge \ch_{\gamma} (\tilde{\nabla}^{E})  
        \nonumber \\ 
        & + 
        \exp (- \frac{1}{2} \Str_{M_\gamma^\bot} \ln f^{\perp}) \wedge \fT \! \hat{A} (\tilde\nabla,\nabla) \wedge \ch_{\gamma} (\tilde{\nabla}^{E}) 
        \nonumber \\ 
        & + 
        \exp (- \frac{1}{2} \Str_{M_\gamma^\bot} \ln f^{\perp}) \wedge \hat{A} (\tilde{\nabla}) \wedge \fT \! \ch_{\gamma} (\tilde{\nabla}^{E}, \nabla^{E}), 
    \end{align}
    where $\fT \! \hat{A} (\fh), \fT \! \ch_{\gamma} (\tilde{\nabla}^{E}, \nabla^{E})$, and $\fT \big( \exp (- \frac{1}{2} \Str_{M_\gamma^\bot} \ln f_{\perp}) \big)$ are given by~\eqref{eq: transgression_form_A_hat},~\eqref{eq: transgression_form_Chern}, and~\eqref{eq: transgression_form_Riem_curv_perp}, respectively. 
\end{lemma}

As the transgression form is locally defined, its values at $\partial M_\gamma$ will only depend on the connections near $\partial M_\gamma$. Since various connections $\tilde \nabla^ \bullet $ are the product connection near the boundary, the resulting transgression form on the boundary is independent of the choices of $\tilde g$ and $\tilde\nabla^E$. In fact, we could also take $\tilde g$ to be a Lorentzian metric with product structure near the boundary and still obtain the same connections (note that the normal bundle of $M_\gamma^\bot$ at $(t,x')$ is the normal bundle of $\Sigma_\gamma$ at $x'$ and will thus only depend on the boundary metric if there is product structure). We will thus write only $\fT\ka$ for $\ft\ka (\tilde \nabla, \nabla,\tilde \nabla^\bot, \nabla^\bot,\tilde\nabla^E,\nabla^E),$ where $\nabla$ and $\nabla^\bot$ are induced by the Levi-Civita connection of the Lorentzian metric $g$ and the quantities with tilde are any that have product structure near $\partial M$ (and agree with the original ones on $\partial M$). 

%
%
%
%
%
%
%
%
%
%
\subsection{The main theorem}

We are now prepared to prove our main theorem.

%
%
%
\ThmGlobalGInd*
%
%
%
\begin{proof}
Let $g_s:=(1-s)\hat g +s\tilde g$, $\<\cdot,\cdot\>_s:=(1-s)\<\cdot,\cdot\>_E+s\<\cdot,\cdot\>'_E$, and $\nabla_s:=(1-s)\nabla+s\tilde\nabla$. To obtain compatible connections, we set
$\tilde\nabla_s:=\frac{1}{2}(\nabla_s+\nabla_s^*),$
where $\nabla_s^*$ denotes the adjoint connection (see (\ref{adjointconn})) with respect to $\<\cdot,\cdot\>_s$. These are all $\gamma$-equivariant. Let $D_s$ denote the twisted Dirac operator for the metric $g_s$ and the connection $\tilde\nabla_s$. Then $D_s$ is a continuous homotopy of equivariant twisted Dirac operators from $\hat D$ to $\tilde D$. As the index is homotopy invariant, we thus have
\[\ind_\gamma(\hat D_{\APS})=\ind_\gamma(\tilde D_{\APS}). \]

Next, we express $\ind_{\gamma} (\tilde{D})$ in terms of the curvature quantities with respect to the connections $\nabla, \nabla^{\perp}, \nabla^{E}$ that are compatible with the Lorentzian metric~\eqref{eq: spacetime_metric} and the hermitian bundle metric on $E$. 
This requires to add the transgression form $\rd \fT \ka$ of $\ka$. 
By~\eqref{eq: int_transgression_form_ASS_integrand}, we have  
\begin{equation*}
    \int_{M_{\gamma}} \tilde{\ka}  
    = 
    \int_{M_{\gamma}} \ka + \int_{\partial M_{\gamma}} \fT \ka.   
\end{equation*}
Inserting this expression in Theorem~\ref{thm: BGVD_ind_thm} yields 
\begin{equation*}
    \gInd (\tilde{D}_{\APS}) 
    = 
    \int_{M_{\gamma}} (2 \pi \ri)^{- \ell} \ri^{-\ell_{\perp}} \ka 
    +  
    \int_{\partial M_{\gamma}} (2 \pi \ri)^{- \ell} \ri^{-\ell_{\perp}} \fT \ka+ 
    \kb .   
\end{equation*}

As shown in Theorem~\ref{indsfdirac} that $\gInd (D_{\APS}) = \sfl_{\gamma}(A)-\tr(\gamma|_{\ker A(1)}) = \gInd (\tilde{D}_{\APS})$. 
Putting everything together, we obtain
\[\gInd(D_{\APS})=\gInd(\hat D_{\APS})=\gInd(\tilde D_{\APS})=\int_{M_{\gamma}} (2 \pi \ri)^{- \ell} \ri^{-\ell_{\perp}} \ka 
    +  
    \int_{\partial M_{\gamma}} (2 \pi \ri)^{- \ell} \ri^{-\ell_{\perp}} \fT \ka+ 
    \kb.   \]
\end{proof}
%
%
%

We conclude this section with some remarks on generalisability. If the twisting bundle $E$ is graded as $E^+\oplus E^-$, the total Dirac index is that for twisting with $E^+$ minus that for twisting with $E^-$, so the index theorem for graded $E$ can be obtained from the one above.

There are various different formulations of the equivariant index theorem in the Riemannian case. As long as the right-hand side does not depend explicitly on the Dirac operator or the metric away from the boundary, it has to coincide with the right-hand side of our theorem by virtue of both sides being equal to the Riemannian index. Thus any such variation also extends to the Lorentzian case. 
For 
instance,~\citet[Theorem 6.4]{Braverman_IndianaUMJ_2019} 
holds in the same way in the Lorentzian setting.

There are variations of the index theorem for non-compact settings: \citet{Braverman_MathZ_2020} describes an index theorem on a spatially non-compact manifold in the presence of a suitable potential, while \citet{Shen_PureApplAanal_2022} describes the index for a timelike non-compact, asymptotically cylindrical spacetime. Corresponding equivariant theorems could probably be obtained by similar methods. We also expect that the main theorem can be extended to the spin$^{\mathrm{c}}$ twisted Dirac operators without great difficulty.

Another question is that of generalising to arbitrary Dirac-type operators.
Our theorem only works for twisted Dirac operators. This is inherent in our method of proof, as we need a Riemannian operator to compare to. A theorem for general Dirac-type operators would probably require an adaptation of 
the~\citet{BSlocal} local index theorem. 

%
%
%
%
%
%
%
%
%
%
\subsection{Examples}
To illustrate how the equivariant index behaves, we now proceed to compute it index in some simple examples.

%
%
%
\begin{ex}[Twisted bundle over $\bbS^1$]
    Let $M=[0,1] \times \bbS^1$ with the Lorentzian product metric $-dt^2 + dx^2$, where $t$ denotes the first coordinate and $x$ denotes the second coordinate (we define coordinates around each point on $\bbS^1$ via $x^{-1}(r)=e^{ir}$).  Let $E=M\times\R^k$ with the standard scalar product and connection
\[\nabla^E=d-itJdx,\]
where $J$ is any self-adjoint matrix. We define a group action of $\bbS^1\subset \C$ on $E$ by
\[z.((t, y),v)=((t, zy),\alpha(z)v),\]
for some unitary representation $\alpha$ of $\bbS^1$ on $\R^n$, i.e. $\bbS^1$ acts on $M$ by spatial rotations and on each fibre by $\alpha$. This preserves all relevant structures. As $\Sigma_t=\{t\} \times \bbS^1$, the corresponding spinor bundle is just $[0,1] \times \Sigma_t\times \R$, with connection given by differentiation and Clifford multiplication given by multiplication with $\ri$. The connection on the twisted spinor bundle of $\Sigma_{t}$ is thus again given by
\[\nabla=d-itJdx.\]
We obtain
\[A_t=\mathsf{c}(\partial_x)\nabla_{\partial_x}=i\partial_x+tJ.\]
The operator $i\partial_x$ on functions $\bbS^1\mapsto \R$ has spectrum $\Z$ with simple eigenvalues. (The eigenfunctions are given by $z\mapsto z^j$.)
We start with the case $k=1$, $J=1$ and $\alpha(z)=z$ (scalar multiplication). The spectral flow of the family $A_t=i\partial_x+t$ is 1.
As the group acts by scalar multiplication, i.e. everything is a single eigenspace, we get for any $z\in \bbS^1$:
\begin{align*}
\ind_z(D_{APS})=\sfl_z(A)=z\sfl(A)=z.
\end{align*}
In particular, we have a non-zero index for any $z\in \bbS^1$. However, the source of this in the index formula varies depending on $z$. If $z=1$, the fact that $A_0$ and $A_1$ have the same spectrum implies that they have the same $\eta$-invariants, so the boundary terms cancel out and we only get a contribution from the interior of $M$. Conversely, if $z\neq 1$, the fixed point set of the action on $M$ is empty, so we get no interior contribution. The index thus comes only from the boundary terms. This show that, while the index as a whole is continuous as a function of the group element, the individual summands in the index theorem need not be continuous. 

Next we look at the case $k=2$, 
$J=\begin{pmatrix}
    1&0\\
    0&-1
\end{pmatrix}$ and $
\alpha(z)=\begin{pmatrix}
    1&0\\
    0&z
\end{pmatrix}$
to show that the equivariant index may indeed be non-zero while the classical index vanishes. We then have
\[A_t=i\partial_x+\begin{pmatrix}
    t&0\\
    0&-t
\end{pmatrix}=(i\partial_x+t)\oplus(i\partial_x-t).\]
As the spectrum of $i\partial_x$ is symmetric, we have
\[\sfl(i\partial_x-t)=\sfl(-{\ri}\partial_x-t)=-\sfl(\partial_x+t)=-1\]
and we obtain
\[\ind_z(D_{APS})=\sfl_z(A)=\sfl(A|_{E_1(\alpha(z))})+z\sfl(A|_{E_z(\alpha(z))})=\sfl(i\partial_x+t)+z\sfl(i\partial_x-t)=1-z.\]
This vanishes for $z=1$ (i.e. in the non-equivariant case), but is non-zero otherwise.
\end{ex}

%
%
%
\begin{ex}[Berger metrics on $\bbS^3$]
    Next we will look at an example of an equivariant index coming from the metric rather than a twist bundle. For this we use the computation of Dirac eigenvalues of Berger spheres from \cite{Hit} to show that we get a non-trivial spectral flow (essentially this is just \cite[section 5]{BSglobal}, only keeping track of the group action).
Consider $\bbS^3$, viewed as the unit sphere of $\C^2\cong \R^4$. The tangent space at each point $x$ can be canonically identified with all vectors that are real orthogonal to $x$. These decompose as
\[T_x\bbS^3\cong i\R x\oplus x^\bot,\]
where now the orthogonal complement is with respect to the complex scalar product. Using this decomposition, the Berger metrics $g_\lambda$ are defined as
\[(g_\lambda)_x(aix+b,a'ix+b'):=\lambda^2aa'+\<b,b'\>_\R\]
for any $a,a'\in \R$ and $b,b'\in x^\bot$. For $\lambda=1$ we obtain the standard metric. As the metric and complex multiplication are preserved by complex unitaries, the action of $U(2)$ on the sphere leaves the Berger metrics invariant. Identifying $\bbS^3$ with $SU(2)$, the space $L^2(\bbS^3)$ can be described via representation theory. $SU(2)$ has exactly one irreducible representation in each dimension, which we will denote by $R_n$ for dimension $n$. By the Peter-Weyl theorem, we have
\[L^2(\bbS^3)\cong \bigoplus\limits_{n=1}^{\infty}R_n\otimes R_n^*.\]
As the tangent bundle of $\bbS^3$ is trivial (identify tangent spaces via the $SU(2)$-action), the spinor bundle of $\bbS^3$ is also trivial, so we obtain
\[L^2(\bbS^3,S(\bbS^3))\cong L^2(\bbS^3)\otimes \R^2\cong \bigoplus\limits_{n=1}^{\infty}R_n\otimes R_n^*\otimes \R^2\]
(with the $SU(2)$ action on $L^2$ corresponding to the $SU(2)$-action on the first tensor factor).

As the Dirac operator is $SU(2)$-equivariant, its action on the first tensor factor is trivial, i.e. $D|_{R_n\otimes R_n^*\otimes \R^2}=1_{R_n}\otimes D_n$ for some $D_n$. The action of $D_n$ on $R_n^*\otimes \R^2$ is computed in \cite[Proposition 3.2.]{Hit}, where it is shown that this has an eigenvalue
\[\frac{\lambda}{2}+\frac{n}{\lambda}\]
of multiplicity 2 (which can never be zero) and single eigenvalues
\[\frac{\lambda}{2}\pm\sqrt{4p(n-p)+\left(\frac{2p-n}{\lambda}\right)^2}\]
for integer $0<p<n$ (those only exist for $n\geq 2$).

In order to obtain a zero eigenvalue, we must have a "$-$" in front of the square root. In that case we have for $n\geq 3$
\[\frac{\lambda}{2}-\sqrt{4p(n-p)+\left(\frac{2p-n}{\lambda}\right)^2}\leq\frac{\lambda}{2}-\sqrt{4p(n-p)}\leq \frac{\lambda}{2}-2\sqrt{2},\]
which is non-zero for $\lambda<4\sqrt{2}$. Thus the only eigenvalue that can vanish for $\lambda\in [0,5]$ occurs for $n=2$ (and thus $p=1$) and is given by
\[\frac{\lambda}{2}-2.\]
Note that the corresponding eigenspace of the Dirac operator is given by $R_2$ (from the first tensor factor). This is just $\C^2$ with the canonical $SU(2)$-action, as the latter is irreducible.
Choose $0<\epsilon<1$ such that no other eigenvalues lie in $[-\epsilon,\epsilon]$ for any $\lambda\in [0,5]$. For $t\in[0,1]$, denote by $A(t)$ the Dirac operator for the Berger metric $g_{5t}$ on $\bbS^3$. Then we have for $\gamma\in SU(2)$ 
\begin{align*}
    \sfl_{\gamma}(A)&=\tr(\gamma|_{E_{[0,\epsilon]}(A(\frac{4+\epsilon}{5}))})-\tr(\gamma|_{E_{[0,\epsilon]}(A(0))})+\tr(\gamma|_{E_{[0,0]}(A(1))})-\tr(\gamma|_{E_{[0,0]}(A(\frac{4+\epsilon}{5}))})\\
    &=\tr(\gamma|_{R_2})-0+0-0\\
    &=\tr(\gamma).
\end{align*}
(Where, by slight abuse of notation, $\gamma$ first denotes the action of $\gamma$ on $L^2(S(\bbS^3))$ and later $\gamma$ itself acting on $\C^2$).

Now let us get back to our Lorentzian setting. Consider $M=[0,1]\times \bbS^3$ with metric $g=-dt^2+g_{5t}$ and $SU(2)$-acting in the natural way on the second factor. Then we have for the equivariant Dirac index on $M$ (for $\gamma\in SU(2)$):
\[\ind_\gamma(D_{APS})=\sfl_\gamma(A(t))=\tr(\gamma).\]
\end{ex}

%
%
%
%
%
%
%
%
%
%
\section*{Acknowledgments}

The authors are indebted to Christian B\"{a}r who suggested the project and made invaluable comments throughout the research.  
We are grateful to Alexander Strohmaier for pointing out and explaining his recent 
work 
on parameter-dependent fundamental solutions. 
O.I. is financially supported by Deutsche Forschungsgemeinschaft (DFG) --- Project Number 546644682: „Feynman-Green-Operatoren f\"{u}r Dirac-Operatoren mit nichtlokalen Randbedingungen“. 
At the earlier stage of this project, he was supported by 
\href{https://www.spp2026.de/projects/2nd-period/boundary-value-problems-and-index-theory-on-riemannian-and-lorentzian-manifolds}{SPP 2026: Geometry at Infinity} funded by DFG.

\bibliography{bibography}
\bibliographystyle{abbrvnat}
\end{document}